\documentclass{amsart}

\usepackage{graphicx,amssymb,amsmath,xcolor,enumerate,latexsym}
\usepackage{hyperref}
\hypersetup{
bookmarks=true,         
pdffitwindow=false,     
pdfstartview={FitH},    
colorlinks=true,      
citecolor=red,
}

\usepackage{enumitem}

\usepackage{float}

\usepackage{mathrsfs}

\makeatletter
\def\l@subsection{\@tocline{2}{0pt}{2.5pc}{5pc}{}}
\makeatother

\usepackage{geometry}
\DeclareSymbolFont{largesymbol}{OMX}{yhex}{m}{n}
\DeclareMathAccent{\Widehat}{\mathord}{largesymbol}{"62}
\newcommand*\di{\mathop{}\!\mathrm{d}}

\def\e{\epsilon}

\numberwithin{equation}{section}              
\newtheorem{theorem}{Theorem}[section]

\newtheorem{lemma}{Lemma}[section]
\newtheorem{proposition}{Proposition}[section]
\newtheorem*{proposition*}{Proposition}
\newtheorem{corollary}{Corollary}[section]
\newtheorem*{corollary*}{Corollary}
\newtheorem{definition}{Definition}[section]
\newtheorem*{definitions*}{Definitions}

\newtheorem*{acknowledgements*}{Acknowledgements}

\newtheorem*{conjecture*}{\bf Conjecture}

\newtheorem*{example*}{\bf Example}
\theoremstyle{remark}
\newtheorem{remark}{\bf Remark}[section]

\begin{document}
\date{}                                     

\author{Cong Wang}
\address[C. Wang]{Department of Mathematics, Harbin Institute of Technology,  Harbin,	150001, P.R. China.}
\email{math\_congwang@163.com}

\author{Yu Gao}
\address[Y. Gao]{Department of Applied Mathematics, The Hong Kong Polytechnic University, Hung Hom, Kowloon, Hong Kong}
\email{mathyu.gao@polyu.edu.hk}

\author{Xiaoping Xue}
\address[X. Xue]{Department of Mathematics, Harbin Institute of Technology, Harbin,	150001, P.R. China.}
\email{xiaopingxue@hit.edu.cn}

\title[Space-time analyticity for Keller-Segel equations]{Optimal decay rates  and space-time analyticity  of  solutions to the Patlak-Keller-Segel equations}

\begin{abstract}
Based on some elementary estimates for the space-time derivatives of the heat kernel,  we use a bootstrapping approach to establish the optimal decay rates for the  $L^q(\mathbb{R}^d)$ ($1\leq q\leq\infty$, $d\in\mathbb{N}$) norm of the space-time derivatives  of solutions to the (modified) Patlak-Keller-Segel equations  with initial data in  $L^1(\mathbb{R}^d)$, which implies the joint space-time analyticity of solutions. When the $L^1(\mathbb{R}^d)$ norm of the initial datum is small, the upper bound for the decay estimates  is global in time, which yields a lower bound on the growth rate of the radius of space-time analyticity in time. As a byproduct, the space analyticity is obtained for any initial data in $L^1(\mathbb{R}^d)$. The decay estimates and space-time analyticity are also established for solutions bounded in both space and time variables.  The results can be extended to a more general class of equations, including the Navier-Stokes equations.
\end{abstract}

\maketitle
{
\hypersetup{linkcolor=blue}
\tableofcontents
}

\section{Introduction}
In this paper, we are going to study the following system of  partial differential equations:
\begin{equation}\label{eq:KS}
\left\{
\begin{aligned}
&\rho_t+\nabla\cdot(\rho\nabla c)=\Delta\rho,\quad x\in\mathbb{R}^d,~~t>0,\\
&c(x,t)=-\frac{1}{d\pi}\int_{\mathbb{R}^d}\ln|x-y|\rho(y,t)\di y.
\end{aligned}\right.
\end{equation} 
When $d=2$, this system is known as the Patlak-Keller-Segel (PKS) model, which is used to describe the collective motion of cells (bacteria or amoebae) attracted by a chemical substance \cite{keller1971model,patlak1953random}.
For $d\in\mathbb{N}$,  system \eqref{eq:KS} is referred to  as the modified Patlak-Keller-Segel (mPKS) equations; see \cite{calvez2006modified,blanchet2008convergence}.

We are going to focus on  the decay rates for the $L^q(\mathbb{R}^d)$ ($1\leq q\leq\infty$) norm of space-time derivatives and the joint space-time analyticity of the mild solutions to the mPKS equations \eqref{eq:KS} with the initial data in  $L^1(\mathbb{R}^d)$. Let $G$ be the heat kernel in $\mathbb{R}^d$ given by
\begin{align*}
G(x,t)=\frac{1}{(4\pi t)^\frac{d}{2}}e^{-\frac{|x|^2}{4t}},\quad x\in\mathbb{R}^d,~~t>0.
\end{align*}
Then, a mild solution $\rho(t)$ of the mPKS equations \eqref{eq:KS} satisfies (see Definition \ref{def:mild})
\begin{align}\label{eq:mildsolution}
\rho(t)=G(\cdot,t)\ast \rho_0-\int_0^tG(\cdot,t-s)\ast \nabla\cdot\big(\rho(s)\nabla c(s)\big)\di s,
\end{align}
where $\rho_0\in L^1(\mathbb{R}^d)$ is the initial datum.
The main results in this paper are as follows: 
\begin{theorem}\label{thm1}
Let $0\leq\rho_0\in L^{1}(\mathbb{R}^d)$, $\beta\in\mathbb{N}^d$ be a multi-index and $k\in\mathbb{N}$ be a nonnegative integer. 
Then, the following two statements hold:
\begin{enumerate}
\item[(i)] If $\|\rho_0\|_{L^1}$ is small enough, then the mild solution $\rho(t)$ to the mPKS equations  \eqref{eq:KS} exists globally, and there exists a constant $M$ independent of time $t$, $\beta$ and $k$ such that the inequality
\begin{align}\label{eq:analytic2}
\left\|D_x^\beta\partial_t^k\rho(t)\right\|_{L^q(\mathbb{R}^d)}\leq M^{|\beta|+k}\left(|\beta|+k\right)^{|\beta|+k}t^{-\frac{|\beta|}{2}-k-\frac{d}{2}(1-\frac{1}{q})}
\end{align}
holds for $1\leq q\leq\infty$, $ t\in(0,\infty)$, and $|\beta|+k>0$.
\item[(ii)] If $\rho(t)$ is a mild solution to  the mPKS equations \eqref{eq:KS} for $t\in[0,T]$ satisfying
\begin{align}\label{eq:condition}
|\rho(x,t)|\leq C,\quad (x,t)\in\mathbb{R}^d\times[0,T]
\end{align}
for some constant $C$, then there exists a constant $M$ depending on $T$, but independent of $\beta$ and $k$, such that 
\eqref{eq:analytic2} holds for $1\leq q\leq\infty$, $ t\in(0,T]$, and $|\beta|+k>0$.
\end{enumerate}
\end{theorem}

As it is known, solutions to some dissipative equations, such as the  Navier-Stokes equations, the surface quasi-geostrophic equation, and the  dissipative Burgers equation are analytic in both space and time variables; see \cite{masuda1967on,kahane1969,foias1989gevrey,giga1983time,dong2008spatial,donghongjie2020jfa,kiselev2008blow} and references therein. When consider the equations in the whole space $\mathbb{R}^d$, most existing results for the analyticity are either about space analyticity or about time analyticity. For a function of real variables, combining its space analyticity and time analyticity does not imply the joint space-time analyticity. (From the Hartogs's theorem or Osgood's lemma, if a function with several complex variables is analytic with respect each variable, then it is analytic with respect to all variables. However, for functions with real variables, we do not have such a good property.)
As far as the authors know, there is no existing result similar to \eqref{eq:analytic2} for the explicit decay estimates of the space-time derivatives of solutions to similar dissipative equations (e.g., the Navier-Stokes equations) in the whole space $\mathbb{R}^d$, which yields the joint space-time analyticity.
Moreover, the decay rate in \eqref{eq:analytic2} is optimal  in the sense that it coincides with the following decay rate for solutions to the heat equation in $\mathbb{R}^d$ (see \eqref{eq:estimate7}):
\begin{align}\label{eq:heatkernel}
\left\|D_x^\beta\partial_t^kG(\cdot,t)\ast f\right\|_{L^q}\leq C^{\frac{|\beta|}{2}+k}\left(|\beta|+k\right)^{\frac{|\beta|}{2}+k+\frac{d}{2}(1-\frac{1}{q})}t^{-\frac{|\beta|}{2}-k-\frac{d}{2}(1-\frac{1}{q})}\|f\|_{L^1},
\end{align}
where $f\in L^1(\mathbb{R}^d)$, $1\leq q\leq\infty$, $\beta\in\mathbb{N}^d$, $k\in\mathbb{N}$, and $C$ is a constant independent of $\beta,k$ and $t$. Comparing with \eqref{eq:heatkernel}, the upper bound $\left(|\beta|+k\right)^{|\beta|+k}$ in \eqref{eq:analytic2}  has some room for improvement.

Instead of proving \eqref{eq:analytic2} directly, we will apply a technique from a recent paper \cite{donghongjie2020jfa}, and change the position of the decay rate $t^{-k}$  (corresponding to time derivatives)  from the right hand side of  \eqref{eq:analytic2} to the left to show
\begin{align}\label{eq:analytic3}
\left\|D_x^\beta\partial_t^k\left(t^k\rho(t)\right)\right\|_{L^{q}}\leq \hat{M}^{|\beta|+k}(|\beta|+k)^{|\beta|+k}t^{-\frac{|\beta|}{2}-\frac{d}{2}(1-\frac{1}{q})}
\end{align}
for $1\leq q\leq\infty$, $|\beta|+k>0$, and some constant $\hat{M}>0$ independent of $\beta$ and $k$.
The above inequality \eqref{eq:analytic3} implies the main result \eqref{eq:analytic2} immediately with $M=\hat{M}+1$ (see the proof of Theorem \ref{thm:analytic}).  The proof of \eqref{eq:analytic3} is based on some new elementary estimates for the space-time derivatives of the heat kernel (see Proposition \ref{pro:heatkernelestimate2}). With these  estimates, we will use an induction for $|\beta|+k$ to study the decay rates of the space-time derivatives of mild solutions to the mPKS equations \eqref{eq:KS}. 
For the case $1<q<\infty$ of \eqref{eq:analytic3}, a bootstrapping approach will be used. More precisely, we will first prove \eqref{eq:analytic3} for $1<q<\frac{d}{d-1}$, and then use the obtained results  to prove the case for $\frac{d}{d-1}\leq q<\frac{d}{d-2}$. Proceed similarly and use the results for $1<q<\frac{d}{d-(k-1)}$ to obtain the  estimate for $\frac{d}{d-(k-1)}\leq q<\frac{d}{d-k}$, $k=2,~3,\cdots, d$.
After $d$ steps, we  will complete the proof of \eqref{eq:analytic3} for $1<q<\infty$.  At last, the estimates \eqref{eq:analytic3} for $q=\infty$ and $q=1$ will be obtained by using the results for $1< q<\infty$. This kind of bootstrapping arguments are mainly due to the restriction brought by the Hardy-Littlewood-Sobolev inequality; see the details in the proofs of Theorem \ref{thm:regularity1} and Theorem \ref{thm:analytic}.
The bootstrapping arguments allow us to avoid the commonly used contraction argument for local smoothing estimates (or hypercontractivity) of solutions to similar dissipative equations; see \cite{bedrossian2014existence,dong2008spatial,giga2002regularizing,sawada2005analyticity}.

In \cite{donghongjie2020jfa}, Dong and Zhang obtained a similar result as  \eqref{eq:analytic3} (with $\beta=0$ and $q=\infty$) to show the time analyticity of solutions to the Navier-Stokes equations in $\mathbb{R}^d$.
For a solution $u$ to the Navier-Stokes equations, under the boundedness condition  $|u(x,t)|\leq C$ for $(x,t)\in\mathbb{R}^d\times[0,1]$, they obtained   (see \cite[Theorem 3.1]{donghongjie2020jfa}):
\[
\sup_{t\in (0,1]}t^k\|\partial_t^ku(\cdot,t)\|_{L^\infty}\leq N^kk^k
\] 
for any $k\in\mathbb{N}$ and some large constant $N$ independent of $k$. Notice that the above inequality was proved in the time interval $[0,1]$. From the proof of \cite[Theorem 3.1]{donghongjie2020jfa}, the constant $N$ actually depends on time $T$ if we consider the same estimate in the time interval $[0,T]$. Therefore, the growth rate of the radius of analyticity in time is not clear. Comparing with the above estimate, we are able to obtain inequality \eqref{eq:analytic3} for solutions to the mPKS equations \eqref{eq:KS} with the same boundedness condition  \eqref{eq:condition} (see Corollary \ref{cor:allinitial}), which implies the space-time analyticity.  
Moreover, for the small initial data, there is no need to assume the boundedness condition \eqref{eq:condition} to obtain space-time analyticity estimate \eqref{eq:analytic2}, and the constant $M$ in \eqref{eq:analytic2} is independent of time $t$. Hence, we can derive from \eqref{eq:analytic2}  that the radius of space-time analyticity $r(t)$ satisfies $r(t)\geq C\min\{\sqrt{t},t\}$ for any $t>0$. For the radius of space analyticity $r_1(t)$, we have $r_1(t)\geq C\sqrt{t}$, and for the radius of time analyticity $r_2(t)$, we have $r_2(t)\geq Ct$ for any $t>0$ (see Remark \ref{rmk:radius}). Note that there is no need to prove \eqref{eq:analytic3} to obtain the following Gevrey type estimate
\begin{align*}
\left\|D_x^\beta\partial_t^k\rho(t)\right\|_{L^q(\mathbb{R}^d)}\leq M^{|\beta|+2k}\left(|\beta|+2k\right)^{|\beta|+2k}t^{-\frac{|\beta|}{2}-k-\frac{d}{2}(1-\frac{1}{q})}.
\end{align*}
Since the time derivative $\partial_t^k$ ($k$ times) of the heat kernel corresponds to space derivatives $\Delta^k$ ($2k$ times), change time derivatives into space derivatives in \eqref{eq:analytic2} and then the above Gevrey type estimate follows. The same idea will be used to obtain the decay estimates for the $L^q(\mathbb{R}^d)$ ($1\leq q \leq\infty$) norm of space-time derivatives of mild solutions to the mPKS equations \eqref{eq:KS} (see Theorem \ref{thm:regularity}).  Similar Gevrey type estimates were obtained  in \cite{dong2008spatial} for the subcritical quasi-geostrophic equations.  However, the above Gevrey type estimate does not imply space-time analyticity.

The results in this paper are obtained by direct calculations about the space-time derivatives of mild solutions to the mPKS equations \eqref{eq:KS}, and some ideas for the proof of Theorem \ref{thm1} are from \cite{giga2002regularizing}, where Giga and Sawada obtained the space analyticity and decay estimates for the space derivatives of solutions to the Navier-Stokes equations. See  \cite{sawada2005analyticity,sawada2008spatial} for similar applications.
For analyticity, another frequently used method based on Fourier analysis was provided by Foias and Teman \cite{foias1989gevrey} to show the time analyticity of solutions to the Navier-Stokes with values in a Gevrey class of functions (for the space variable). For the generalizations of their methods, see \cite{ferrari1998gevrey,biswas2014maximal,kukavica2009radius,lemarie2000,oliver2000remark} and references therein. There are also lots of studies for analyticity of solutions to some equations with fractional dissipation  \cite{dong2008spatial,dong2009optimal,li2011one,gao2020global}, and equations in domains with boundaries \cite{masuda1967on,kinderlehrer1978analyticity,komatsu1980global,giga1983time,camliyurt2018gevrey,camliyurt2020analyticity}.

For the well-posedness results of the PKS equations \eqref{eq:KS} (i.e., $d=2$), one can refer to \cite{blanchet2006two,blanchet2008infinite,blanchet2012functional,bedrossian2014existence,wei2018global}. See \cite{calvez2006modified,blanchet2008convergence} for the well-posedness of the mPKS equations \eqref{eq:KS} for any dimensional $d\in\mathbb{N}$.
If $\rho(x,t)$ is a solution to the mPKS equations \eqref{eq:KS}, then $\rho_\lambda$ given by  $\rho_\lambda(x,t)=\lambda^d\rho(\lambda x,\lambda^2 t)$ is also a solution. Moreover, the $L^1(\mathbb{R}^d)$ norm is invariant under the above scaling. 
When $d=2$, it is well known that there exists a unique global solution of the PKS equations when the initial datum satisfies $\|\rho_0\|_{L^1}\leq 8\pi$ and the solution will blow up in finite time if $\|\rho_0\|_{L^1}>8\pi$; see \cite{jager1992explosions,blanchet2006two,dolbeault2009two}. When come to global well-posedness, most of the results were obtained under the assumption $\int_{\mathbb{R}^2}\left(|x|^2+\log{\rho_0}\right)\rho_0<+\infty$ for the nonnegative initial datum $\rho_0$, and with this assumption, one can take advantage of the free energy of the PKS equations. In \cite{wei2018global},  Wei removed the above assumption and used some monotonicity formulas of nonnegative mild solutions to obtain a global solution to the PKS equations for any $L^1(\mathbb{R}^2)$ initial datum $\rho_0$ satisfying $\|\rho_0\|_{L^1}\leq 8\pi$. His method can be applied to the mPKS equations \eqref{eq:KS} for any dimension  $d\in\mathbb{N}$ and obtain a global solution when $\|\rho_0\|_{L^1}\leq 2d^2\pi$.

The rest of this paper is organized as follows. In Section \ref{sec:pre}, we will show some elementary space-time estimates for the heat kernel. In Section \ref{sec:local}, we will prove local existence and uniqueness of mild solutions for any $L^1(\mathbb{R})$ initial data. When the mass of initial datum is small enough, we will obtain a global mild solution. The instantaneous regularization and decay rates of the space-time derivatives will also be obtained in this section. The main results of Theorem \ref{thm1} will be proved in Section \ref{sec:analytic}.

\section{Properties of the heat kernel}\label{sec:pre}
In this section, we will provide some estimates for the heat kernel.
We have the following estimates for spatial derivatives (see \cite{carrillo2008asymptotic,dong2008spatial,gao2020global} for some similar results):
\begin{lemma}\label{lmm:heatkernelestimate1}
Let $f\in L^p(\mathbb{R}^d)$ for some $p\geq 1$. Assume $0\neq \beta\in\mathbb{N}^d$ is a multi-index. Then for $1\leq q\leq \infty$, we have
\begin{equation}
\begin{aligned}\label{eq:estimate1}
\left\|D_x^{\beta}G(\cdot,t)\right\|_{L^q}\leq C_{0}^{\frac{|\beta|}{2}}\left(\frac{|\beta|}{2}\right)^{\frac{|\beta|}{2}+\frac{d}{2}(1-\frac{1}{q})}t^{-\frac{|\beta|}{2}-\frac{d}{2}(1-\frac{1}{q})},
\end{aligned}
\end{equation}
\begin{equation}
\begin{aligned}\label{eq:estimate2}
\left\|D_x^{\beta}\nabla G(\cdot,t)\right\|_{L^q}\leq C_{0}^{\frac{|\beta|}{2}}\left(\frac{|\beta|}{2}\right)^{\frac{|\beta|+1}{2}+\frac{d}{2}(1-\frac{1}{q})}t^{-\frac{|\beta|+1}{2}-\frac{d}{2}(1-\frac{1}{q})},
\end{aligned}
\end{equation}
and for $1\leq p\leq q\leq\infty$, we have
\begin{equation}
\begin{aligned}\label{eq:estimate3}
\left\|\left[D_x^{\beta}G(\cdot,t)\right]\ast f\right\|_{L^q}\leq C_{0}^{\frac{|\beta|}{2}}\left(\frac{|\beta|}{2}\right)^{\frac{|\beta|}{2}+\frac{d}{2}(\frac{1}{p}-\frac{1}{q})}t^{-\frac{|\beta|}{2}-\frac{d}{2}(\frac{1}{p}-\frac{1}{q})}\|f\|_{L^p},
\end{aligned}
\end{equation}
\begin{equation}
\begin{aligned}\label{eq:estimate4}
\left\|\left[D_x^{\beta}\nabla G(\cdot,t)\right]\ast f\right\|_{L^q}\leq C_{0}^{\frac{|\beta|}{2}}\left(\frac{|\beta|}{2}\right)^{\frac{|\beta|+1}{2}+\frac{d}{2}(\frac{1}{p}-\frac{1}{q})}t^{-\frac{|\beta|+1}{2}-\frac{d}{2}(\frac{1}{p}-\frac{1}{q})}\|f\|_{L^p}.
\end{aligned}
\end{equation}
Moreover,
\begin{align}\label{eq:key}
\lim_{t\to0}t^{\frac{d}{2}(\frac{1}{p}-\frac{1}{q})}\|G(\cdot,t)\ast f\|_{L^q}=0,\quad \forall ~q>p.
\end{align}
In particularly, for $p=1$ and $q=\frac{2d}{2d-1}$, we obtain from \eqref{eq:key} that
\begin{align}\label{eq:key1}
	\lim_{t\to0}t^{\frac{1}{4}}\|G(\cdot,t)\ast f\|_{L^{\frac{2d}{2d-1}}}=0.
\end{align}
\end{lemma}
\begin{proof}
Let $y=xt^{-\frac{1}{2}}$. From the definition of the heat kernel, we have 
\begin{align*}
G(x,t)=t^{-\frac{d}{2}}G(y,1),\quad\nabla G(x,t)=t^{-\frac{d+1}{2}}\nabla G(y,1)
\end{align*}
and 
\begin{align*}
D_x^\beta G(x,t)=t^{-\frac{|\beta|}{2}-\frac{d}{2}}D_y^\beta G(y,1),\quad 
D_x^\beta\nabla G(x,t)=t^{-\frac{|\beta|+1}{2}-\frac{d}{2}}D_y^\beta\nabla G(y,1).
\end{align*}
Hence, for $1\leq q\leq\infty$, we have
\begin{align}\label{eq:e1}
\left\|D_x^\beta G(\cdot,t)\right\|_{L^q}=t^{-\frac{|\beta|}{2}-\frac{d}{2}\left(1-\frac{1}{q}\right)}\left\|D_y^\beta G(\cdot,1)\right\|_{L^q}
\end{align}
and 
\begin{align}\label{eq:e2}
\left\|D_x^\beta\nabla G(\cdot,t)\right\|_{L^q}=t^{-\frac{|\beta|+1}{2}-\frac{d}{2}\left(1-\frac{1}{q}\right)}\left\|D_y^\beta\nabla G(\cdot,1)\right\|_{L^q},
\end{align}
which imply
\begin{align}\label{eq:e3}
\left\|D_x^{\beta^{(i)}}G\left(\cdot,\frac{1}{|\beta|}\right)\right\|_{L^q}=\left(\frac{1}{|\beta|}\right)^{-\frac{1}{2}-\frac{d}{2}\left(1-\frac{1}{q}\right)}\left\|D_y^{\beta^{(i)}}G(\cdot,1)\right\|_{L^q}
\end{align}
and 
\begin{align}\label{eq:e4}
\left\|D_x^{\beta^{(i)}}\nabla G\left(\cdot,\frac{1}{|\beta|}\right)\right\|_{L^q}=\left(\frac{1}{|\beta|}\right)^{-1-\frac{d}{2}\left(1-\frac{1}{q}\right)}\left\|D_y^{\beta^{(i)}}\nabla G(\cdot,1)\right\|_{L^q},
\end{align}
where $\beta^{(i)}$ $(i=1,2,\cdots,|\beta|)$ satisfies $\beta^{(i)}\in\mathbb{N}^d$, $\left|\beta^{(i)}\right|=1$ and $\sum_{i=1}^{|\beta|}\beta^{(i)}=\beta$. By the Fourier transform, there holds:
\begin{align*}
D_x^\beta G(\cdot,1)=D_x^{\beta^{(1)}}G\left(\cdot,\frac{1}{|\beta|}\right)\ast D_x^{\beta^{(2)}}G\left(\cdot,\frac{1}{|\beta|}\right)\ast\cdots\ast D_x^{\beta^{(|\beta|)}}G\left(\cdot,\frac{1}{|\beta|}\right)
\end{align*}
and 
\begin{align*}
D_x^\beta\nabla_xG(\cdot,1)=D_x^{\beta^{(1)}}G\left(\cdot,\frac{1}{|\beta|}\right)\ast D_x^{\beta^{(2)}}G\left(\cdot,\frac{1}{|\beta|}\right)\ast\cdots\ast D_x^{\beta^{(|\beta|)}}\nabla G\left(\cdot,\frac{1}{|\beta|}\right).
\end{align*}
Combining Young's inequality, \eqref{eq:e3} and \eqref{eq:e4}, we obtain inequalities \eqref{eq:estimate1} and \eqref{eq:estimate2}. Then we obtain \eqref{eq:estimate3} and \eqref{eq:estimate4} from \eqref{eq:estimate1} and \eqref{eq:estimate2} by using Young's inequality again. Moreover, from the above calculation, we can choose $C_0$ as 
\[
C_0:=2\left\{\sup_{1\leq p\leq \infty}\|\nabla G(\cdot,1)\|_{L^p},\sup_{1\leq p\leq \infty}\|\nabla^2 G(\cdot,1)\|_{L^p}\right\}.
\]

Next, we prove \eqref{eq:key} by the method in \cite[Lemma 2.1]{carrillo2008asymptotic}, which proves the same results for the  fractional heat kernel when $d=2$.
Choose a sequence of functions $\{f_n\}_{n\in \mathbb{N}}$ satisfying $f_n\in L^p(\mathbb{R}^d)\cap L^q(\mathbb{R}^d)$  and $\|f_n-f\|_{L^p}\to 0$ as $n\to \infty$. From Young's inequality, we have
\[
\left\|G(\cdot,t)\ast f_n\right\|_{L^q}\leq \|f_n\|_{L^q},
\]
which implies
\begin{align*}
\lim_{t\to 0}t^{\frac{d}{2}(\frac{1}{p}-\frac{1}{q})}\|G(\cdot,t)\ast f_n\|_{L^q}=0,\quad q>p.
\end{align*}
By Young's inequality again, we have
\begin{align}\label{eq:convergence}
\left\|t^{\frac{d}{2}(\frac{1}{p}-\frac{1}{q})}G(\cdot,t)\ast f_n-t^{\frac{d}{2}(\frac{1}{p}-\frac{1}{q})}G(\cdot,t)\ast f\right\|_{L^q}\leq C\|f_n-f\|_{L^p}\to 0~\textrm{ as }~n\to\infty.
\end{align}
The above convergence is independent of $t$. From this, we obtain
\begin{align*}
\lim_{t\to 0}t^{\frac{d}{2}(\frac{1}{p}-\frac{1}{q})}\|G(\cdot,t)\ast f\|_{L^q}=0,\quad q>p.
\end{align*}

\end{proof}

For the space-time derivatives of heat kernel $G$, we have the following elementary estimates:
\begin{proposition}\label{pro:heatkernelestimate2}
Let $f\in L^1(\mathbb{R}^d)$, $k\in\mathbb{N}$, $\beta \in\mathbb{N}^d$ and $|\beta|+k>0$. Assume $m\in\mathbb{N}$ satisfying $0\leq m\leq k$. Then, there exists a constant $M_0>0$  independent of $\beta$ and $k$ such that the following inequalities hold for $1\leq q\leq +\infty$:
\begin{align}\label{eq:estimate5}
\left\|D_x^\beta\partial_t^k\left[t^mG(\cdot,t)\right]\right\|_{L^q}\leq M_0^{\frac{|\beta|}{2}+k}\left(|\beta|+k\right)^{\frac{|\beta|}{2}+k+\frac{d}{2}(1-\frac{1}{q})}t^{-\frac{|\beta|+2(k-m)}{2}-\frac{d}{2}(1-\frac{1}{q})},
\end{align}
\begin{align}\label{eq:estimate6}
\left\|D_x^\beta\partial_t^k\left[t^m\nabla G(\cdot,t)\right]\right\|_{L^q}\leq M_0^{\frac{|\beta|}{2}+k}\left(|\beta|+k\right)^{\frac{|\beta|+1}{2}+k+\frac{d}{2}(1-\frac{1}{q})}t^{-\frac{|\beta|+2(k-m)+1}{2}-\frac{d}{2}(1-\frac{1}{q})},
\end{align}
\begin{align}\label{eq:estimate7}
\left\|D_x^\beta\partial_t^k\left[t^mG(\cdot,t)\ast f\right]\right\|_{L^q}\leq M_0^{\frac{|\beta|}{2}+k}\left(|\beta|+k\right)^{\frac{|\beta|}{2}+k+\frac{d}{2}(1-\frac{1}{q})}t^{-\frac{|\beta|+2(k-m)}{2}-\frac{d}{2}(1-\frac{1}{q})}\|f\|_{L^1},
\end{align}
\begin{align}\label{eq:estimate8}
\left\|D_x^\beta\partial_t^k\left[t^m\nabla G(\cdot,t)\ast f\right]\right\|_{L^q}\leq M_0^{\frac{|\beta|}{2}+k}\left(|\beta|+k\right)^{\frac{|\beta|+1}{2}+k+\frac{d}{2}(1-\frac{1}{q})}t^{-\frac{|\beta|+2(k-m)+1}{2}-\frac{d}{2}(1-\frac{1}{q})}\|f\|_{L^1}. 
\end{align}
\end{proposition}
\begin{proof}
\textbf{Step 1:} We prove the case of $m=0$.

We first prove \eqref{eq:estimate6}.
Since $\Delta=\sum_{i=1}^d\partial_{x_i}^2$ and $\left\|D_x^{\beta}\partial_t^{k}\nabla G(\cdot,t)\right\|_{L^q}= \left\|D_x^{\beta}\Delta^{k}\nabla G(\cdot,t)\right\|_{L^q},$
combining \eqref{eq:estimate2} gives
\begin{equation}\label{eq:estimate55}
\begin{aligned}
\left\|D_x^{\beta}\partial_t^{k} \nabla G(\cdot,t)\right\|_{L^q}&\leq d^k\sup_{\substack{\beta\leq \tilde{\beta},\\|\tilde{\beta}|=|\beta|+2k}}\|D_x^{\tilde{\beta}}\nabla G(\cdot,t)\|_{L^q}\\
&\leq (dC_{0})^{\frac{|\beta|+2k}{2}}\left(\frac{|\beta|+2k}{2}\right)^{\frac{|\beta|+2k+1}{2}+\frac{d}{2}(1-\frac{1}{q})}t^{-\frac{|\beta|+2k+1}{2}-\frac{d}{2}(1-\frac{1}{q})}.
\end{aligned}
\end{equation}
Therefore, inequality \eqref{eq:estimate6} holds for constant $\tilde{M}_0=dC_0$. By similar arguments, we obtain \eqref{eq:estimate5} for $m=0$.

Inequalities \eqref{eq:estimate7} and \eqref{eq:estimate8} follow from \eqref{eq:estimate5}, \eqref{eq:estimate6}, and Young's inequality.

\textbf{Step 2:} We prove for the case: $0< m\leq k$. We will also prove \eqref{eq:estimate6} first.
We have
\begin{align*}
\left\|D_x^\beta\partial_t^k\left[t^m\nabla G(\cdot,t)\right]\right\|_{L^q}&=\left\|\sum_{j=0}^k\binom{k}{j}(\partial_t^{k-j}t^m)D_x^\beta\partial_t^j\nabla G(\cdot,t)\right\|_{L^q}\\
&=\left\|\sum_{j=k-m}^k\binom{k}{j}(\partial_t^{k-j}t^m)D_x^\beta\partial_t^j \nabla G(\cdot,t)\right\|_{L^q}.
\end{align*}
Using the fact
\[
\partial_t^{k-j}t^m=m(m-1)\cdots(m-k+j+1)t^{m-k+j}\leq \left(|\beta|+k\right)^{k-j}t^{m-k+j},
\]
combining Step 1 gives
\begin{align*}
&\left\|D_x^\beta\partial_t^k\left[t^m \nabla G(\cdot,t)\right]\right\|_{L^q}\\
\leq&\sum_{j=k-m}^{k}\binom{k}{j} \left(|\beta|+k\right)^{k-j} t^{m-k+j} \tilde{M}_0^{\frac{|\beta|}{2}+j}\left(|\beta|+k\right)^{\frac{|\beta|+1}{2}+j+\frac{d}{2}(1-\frac{1}{q})}t^{-\frac{|\beta|+2j+1}{2}-\frac{d}{2}(1-\frac{1}{q})}\\
\leq&\sum_{j=k-m}^{k}\binom{k}{j} \tilde{M}_0^{\frac{|\beta|}{2}+k}\left(|\beta|+k\right)^{\frac{|\beta|+1}{2}+k+ \frac{d}{2}(1-\frac{1}{q})}t^{-\frac{|\beta|+2(k-m)+1}{2}-\frac{d}{2}(1-\frac{1}{q})}\\
\leq &2^k \tilde{M}_0^{\frac{|\beta|}{2}+k}\left(|\beta|+k\right)^{\frac{|\beta|+1}{2}+k+\frac{d}{2}(1-\frac{1}{q}) }t^{-\frac{|\beta|+2(k-m)+1}{2}-\frac{d}{2}(1-\frac{1}{q})},
\end{align*}
where we used the fact $\sum_{j=k-m}^{k}\binom{k}{j}\leq \sum_{j=1}^{k}\binom{k}{j} =2^k$ in the last inequality. Choose $M_0=2\tilde{M}_0$ and inequality \eqref{eq:estimate6} holds for any $0\leq m\leq k$.

Then, we can obtain \eqref{eq:estimate5}, \eqref{eq:estimate7}, and \eqref{eq:estimate8} for any $0\leq m\leq k$ similarly as in Step 1.
\end{proof}

\section{Local well-posedness and instantaneous regularity}\label{sec:local}
In this section, we are going to show local existence, uniqueness and instantaneous regularity of mild solutions to the mPKS equations  \eqref{eq:KS} with initial data $\rho_0\in L^1(\mathbb{R}^d)$. 
We will study the mild solutions in the following space:
\begin{align}\label{eq:spaces}
X_T:=\left\{f\in C_b\left((0,T];L^{1}(\mathbb{R}^d)\right),~~\sup_{t\in(0,T]}t^{\frac{1}{4}}\|f(t)\|_{L^{\frac{2d}{2d-1}}}<\infty\right\}
\end{align}
with norm 
\[
\|f\|_{X_T}:=\max\left\{\sup_{t\in(0,T]}\|f(t)\|_{L^1},\quad\sup_{t\in(0,T]}t^{\frac{1}{4}}\|f(t)\|_{L^{\frac{2d}{2d-1}}} \right\}.
\]
Then, space $(X_T,\|\cdot\|_{X_T})$ is a Banach space. The definition of the mild solutions is given by
\begin{definition}[Mild solutions]\label{def:mild}
We call $\rho\in X_T$ a mild solution to the mPKS equations  \eqref{eq:KS} with initial datum $\rho_0\in L^1(\mathbb{R}^d)$ if $\rho$ satisfies the following Duhamel integral equation in  $X_T$:
\begin{align}\label{eq:mild}
\rho(t)=G(\cdot,t)\ast\rho_0-\int_0^t\nabla G(\cdot,t-s)\ast\left[\rho(s)\nabla c(s)\right]\di s,\quad t\in[0,T].
\end{align}
If equation \eqref{eq:mild} holds for any $T>0$, then we call $\rho$ a global mild solution.
\end{definition}
We have the following theorem:
\begin{theorem}\label{thm:contraction}
Let $0\leq \rho_0\in L^{1}(\mathbb{R}^d)$. There exists  a constant $\theta>0$ small enough such that if $T>0$ satisfies
\begin{align}\label{eq:important}
\sup_{t\in(0,T]}t^{\frac{1}{4}}\|G(\cdot,t)\ast\rho_0\|_{L^{\frac{2d}{2d-1}}}\leq \theta,
\end{align}
then there is a unique mild solution $\rho\in X_{T}$ to the mPKS equations  \eqref{eq:KS} in the following set:
\[
X_{T}^\theta:=\left\{u\in X_{T}:~~\sup_{t\in(0,T]}t^{\frac{1}{4}}\|u(t)\|_{L^{\frac{2d}{2d-1}}}\leq 2\theta\right\}.
\]
We also have $\rho\in C\left([0,T]; L^{1}(\mathbb{R}^d)\right)$ and $\rho(x,0)=\rho_0(x)$, for $ x\in\mathbb{R}^d$. 

Moreover, if the total mass $\|\rho_0\|_{L^1}$ is small enough, then for any $T>0$ we have
\begin{align}\label{eq:assumptiona0}
\sup_{t\in(0,T]}t^{\frac{d}{2}(1-\frac{1}{q})}\|G(\cdot,t)\ast\rho_0\|_{L^{q}}\leq \theta,\quad 1\leq q\leq\infty,
\end{align}
and there is a unique global mild solution.
\end{theorem}
The local well-posedness of mild solutions for some dissipative equations similar to the mPKS equations is well studied; see, e.g. \cite{weissler1980local,kato1984strongl,biler1995cauchy,biler2010blowup,carrillo2008asymptotic}. Since some estimates in the proof of this theorem are very useful in the rest of this paper, we are going to provide a complete proof here. 
\begin{proof}

Define the operator $S$:
\begin{align}\label{eq:contractionmap}
S\rho:=G(\cdot,t)\ast\rho_0-\int_0^t\nabla G(\cdot,t-s)\ast\left[\rho(s)\nabla c(s)\right]\di s. 
\end{align}
We only need to prove that $S:X_T^\theta\to X_T^\theta$ is a contraction mapping for $\theta$ small enough.

\textbf{Step 1}: We show that $S\rho \in X_{T}^\theta$ for $\rho\in X_{T}^\theta$. Suppose $c(x,s)=-\frac{1}{d\pi}\int_{\mathbb{R}^d}\ln|x-y|\rho(y,s)\di y$. Set $S\rho=\rho_1+\rho_2$ with
\begin{align}\label{eq:rho12}
\rho_1(x,t):=G(\cdot,t)\ast\rho_0,\quad \rho_2(x,t):=-\int_0^t\nabla G(\cdot,t-s)\ast[\rho(s)\nabla c(s)]\di s.
\end{align}
We first prove that $S\rho\in C_b\left((0,T];L^{1}(\mathbb{R}^d)\right)$.
Notice that $\rho_1$ is the solution of the heat equation with initial datum $\rho_0$, and we have $\rho_1\in C_b\left((0,T];L^{1}(\mathbb{R}^d)\right)$.

From Hardy-Littlewood-Sobolev inequality, we have
\begin{align*}
\|\nabla c(t)\|_{L^{2d}}=\frac{1}{d\pi}\left\||x|^{-1}\ast \rho(t)\right\|_{L^{2d}}\leq C\|\rho(t)\|_{L^{\frac{2d}{2d-1}}}.
\end{align*}
Combining with H\"older's inequality gives
\begin{align}\label{eq:HLSI}
\|\rho(x)\nabla c(s)\|_{L^{1}}\leq \|\rho(s)\|_{L^{\frac{2d}{2d-1}}}\|\nabla c(s)\|_{L^{2d}}\leq C\|\rho(s)\|_{L^{\frac{2d}{2d-1}}}^2
\end{align}
From \eqref{eq:HLSI} and \eqref{eq:estimate4} with $p=q=1$, we have
\begin{equation}\label{eq:boundedd2}
\begin{aligned}
\|\rho_2(t)\|_{L^{1}}\leq &\int_0^t\left\|\nabla G(\cdot,t-s)\ast[\rho(s)\nabla c(s)]\right\|_{L^{1}}\di s\leq C\int_0^t(t-s)^{-\frac{1}{2}} \left\|\rho(s)\nabla c(s)\right\|_{L^{1}}\di s\\
\leq &C\theta^2\int_0^t(t-s)^{-\frac{1}{2}}s^{-\frac{1}{2}}\di s\leq C\mathcal{B}\left(\frac{1}{2},\frac{1}{2}\right)\theta^2,
\end{aligned}
\end{equation} 
where $\mathcal{B}(\frac{1}{2},\frac{1}{2})$ is a Beta function $\mathcal{B}(a,b)=\int_0^1s^{a-1}(1-s)^{b-1}\di s$ with $a=\frac{1}{2}$ and $b=\frac{1}{2}$. 
This estimate shows that $\rho_2(t)\in L^{1}(\mathbb{R}^d)$. For the time continuity of $\rho_2(t)$, the proof is routine and tedious, and hence we put it into appendix \ref{app:continuity}.

Next, we estimate $t^{\frac{1}{4}}\|S\rho(t)\|_{L^{\frac{2d}{2d-1}}}$. For $p=1$ and $q=\frac{2d}{2d-1}$ in \eqref{eq:estimate4},  combine \eqref{eq:HLSI} and we have
\begin{align*}
\|\nabla G(\cdot,t-s)\ast [\rho(s)\nabla c(s)]\|_{L^{\frac{2d}{2d-1}}}&\leq C(t-s)^{-\frac{3}{4}}\|\rho(s)\nabla c(s)\|_{L^{1}}\\
&\leq C(t-s)^{-\frac{3}{4}}s^{-\frac{1}{2}}\theta^2,
\end{align*}
which implies
\begin{multline}
\sup_{0<t\leq T}t^{\frac{1}{4}}\int_0^t \|\nabla G(\cdot,t-s)\ast  [\rho(s)\nabla c(s)]\|_{L^{\frac{2d}{2d-1}}}\di s \\
\leq C\theta^2 \cdot \sup_{0<t\leq T}t^{\frac{1}{4}}\int_0^t (t-s)^{-\frac{3}{4}}s^{-\frac{1}{2}} \di s \leq C\mathcal{B}\left(\frac{1}{2},\frac{1}{4}\right)\theta^2.
\end{multline}
Here, the constant $C$ is independent of $T$. Hence, there exists $\theta>0$ small enough such that
\begin{align}\label{eq:in1}
\sup_{0<t\leq T}t^{\frac{1}{4}}\|S\rho(t)\|_{L^{\frac{2d}{2d-1}}}
\leq \sup_{0<t\leq T}t^{\frac{1}{4}}\|G(\cdot,t)\ast\rho_0\|_{L^{\frac{2d}{2d-1}}}+C\mathcal{B}\left(\frac{1}{2},\frac{1}{4}\right)\theta^2<2\theta.
\end{align}

\textbf{Step 2.} We are going to show that $S$ is a contraction mapping. For $\tilde{\rho},~\hat{\rho}\in X_T^\theta$, assume
\[
\tilde{c}(x,s)=-\frac{1}{d\pi}\int_{\mathbb{R}^d}\ln|x-y|\tilde{\rho}(y,s)\di y, ~~s\in[0,T],
\]
and $\hat{c}$ is generated by $\hat{\rho}$ by the same relation.
To prove $S$ is a contraction mapping, we have
\begin{equation}\label{eq:contraction}
\begin{aligned}
\|S\tilde{\rho}-&S\hat{\rho}\|_{X_T}\leq \left\|\int_0^t\nabla G(\cdot,t-s)\ast\left[(\tilde{\rho}(s)-\hat\rho(s))\nabla \tilde{c}(s)\right]\di s\right\|_{X_T}\\
&+ \left\|\int_0^t\nabla G(\cdot,t-s)\ast\left[\hat\rho(s)\nabla (\tilde{c}(s)-\hat{c}(s))\right]\di s\right\|_{X_T}.
\end{aligned}
\end{equation}
By the same estimates as in step 1, we have the following estimates for the first term in \eqref{eq:contraction}:
\begin{align*}
\sup_{0\leq t\leq T}\left\|\int_0^t\nabla G(\cdot,t-s)\ast\left[(\tilde{\rho}(s)-\hat\rho(s))\nabla \tilde{c}(s)\right]\di s\right\|_{L^{1}}\leq C\theta\left\|\tilde{\rho}(s)-\hat\rho(s)\right\|_{X_T},
\end{align*}
and
\begin{align*}
\sup_{0< t\leq T}t^{\frac{1}{4}}\left\|\int_0^t\nabla G(\cdot,t-s)\ast\left[(\tilde{\rho}(s)-\hat\rho(s))\nabla \tilde{c}(s)\right]\di s\right\|_{L^{\frac{2d}{2d-1}}}\leq C\theta\left\|\tilde{\rho}(s)-\hat\rho(s)\right\|_{X_T}
\end{align*}
for some constant $C$ independent of $T$.
Similar estimates can also be obtained for the second term in the right hand of inequality \eqref{eq:contraction}, and we obtain
\begin{align}\label{eq:contractionlast}
\left\|S\tilde\rho -S\hat\rho_2 \right\|_{X_T}\leq C\theta\left\|\tilde{\rho}(s)-\hat\rho(s)\right\|_{X_T}.
\end{align}
Hence, $S$ is a contraction mapping for small enough $\theta>0$.

\textbf{Step 3.} In this step, we are going to prove the time continuity at $t=0$. By the property of the solution to the heat equation, we have
\begin{align*}
\lim_{t\rightarrow 0}\|G(\cdot,t)\ast\rho_0-\rho_0\|_{L^1}=0,
\end{align*} 
hence we only need to show
\begin{align}\label{eq:timecontinuity}
\lim_{t\rightarrow 0}\left\|\int_0^t\nabla G(\cdot, t-s)\ast[\rho(s)\nabla c(s)]\di s\right\|_{L^{1}}=0.
\end{align}
Let $\rho$ and $\bar{\rho}$ be the solution in $X_T^\theta$ and $X_{\bar{T}}^{\bar{\theta}}$ respectively with $\bar{\theta}<\theta$ and $\bar{T}<T$. Similarly to \eqref{eq:contractionlast}, we obtain
\begin{align*}
\left\|S\rho -S\bar{\rho} \right\|_{X_{\bar{T}}}\leq C\theta\left\|\rho-\bar{\rho}\right\|_{X_{\bar{T}}}.
\end{align*}
this shows $\bar{\rho}(t)=\rho(t)$ for $0<t\leq\bar{T}$. Due to \eqref{eq:key1}, as $\bar{T}\rightarrow 0$, we could choose $\bar{\theta}\rightarrow 0$.  Combining \eqref{eq:boundedd2} for $\bar{\theta}\rightarrow 0$, we have \eqref{eq:timecontinuity}, and the equality $\rho(x,0)=\rho_0(x)$ follows.

\textbf{Step 4.}  Applying Young's inequality and \eqref{eq:e1} for $\beta=0$, we have
\begin{align*}
t^{\frac{d}{2}(1-\frac{1}{q})}\|G(\cdot,t)\ast \rho_0\|_{L^{q}}\leq C\|\rho_0\|_{L^1},\quad 1\leq q\leq \infty.
\end{align*}
When $C\|\rho_0\|_{L^1}\leq\theta$,  inequality \eqref{eq:assumptiona0} holds for any $T>0$. Condition \eqref{eq:important}  corresponds to $q=\frac{2d}{2d-1}$ in \eqref{eq:assumptiona0}. Therefore, from the above steps, there exists a unique global mild solution for the initial datum $\rho_0$ with small total mass. 

\end{proof}
For any initial datum  $\rho_0\in L^1(\mathbb{R}^d)$ and $\theta>0$, due to \eqref{eq:key1}, there exists $T>0$ such that \eqref{eq:important} holds, which implies local existence and uniqueness of a mild solution. Based on \eqref{eq:key1}, the time $T$ for inequality \eqref{eq:important} usually depends on $\theta$. 
The small constant $\theta$ in Theorem \ref{thm:contraction} plays a very import role in this paper.  
In the rest of this paper, we will assume that the initial datum $\rho_0\in L^1(\mathbb{R}^d)$ satisfies \eqref{eq:assumptiona0} for some small constant $\theta$, i.e.,
\begin{align}\label{eq:assumptiona}
\sup_{t\in(0,T]}t^{\frac{d}{2}(1-\frac{1}{q})}\|G(\cdot,t)\ast\rho_0\|_{L^{q}}\leq \theta,\quad 1\leq q\leq\infty \tag{$\mathscr{A}$}
\end{align}
for some time $T>0$. From Theorem \ref{thm:contraction}, the assumption \eqref{eq:assumptiona} holds for any time $T>0$ if $\|\rho_0\|_{L^1}$ is small enough, and hence global existence and uniqueness of mild solutions for these small initial data are obtained.

\begin{remark}
In the rest of this paper, we will prove our results with the assumption \eqref{eq:assumptiona} to emphasize the results for small initial data. However, except for the time analyticity (see Step 3 in the proof of Theorem \ref{thm:analytic} and Remark \ref{rmk:assumptiona}), all the other results  can be obtained without the assumption \eqref{eq:assumptiona} following the same methods.
\end{remark}

We have the following  $L^q(\mathbb{R}^d)$ ($1\leq q\leq\infty$) estimates (i.e., hypercontractivity) for the mild solutions:
\begin{theorem}[$L^q(\mathbb{R}^d)$ estimate]\label{thm:regularity1}
Let $0\leq \rho_0\in L^{1}(\mathbb{R}^d)$ and $T>0$ satisfy \eqref{eq:assumptiona} for some  constant $\theta$. Then, the mild solution $\rho(t)$ obtained by Theorem \ref{thm:contraction} belongs to $L^q(\mathbb{R}^d)$ for any $1\leq q\leq +\infty$, and we have
\begin{align}\label{eq:q}
\|\rho(t)\|_{L^q(\mathbb{R}^d)}\leq A\theta t^{-\frac{d}{2}(1-\frac{1}{q})},\quad  1\leq q\leq\infty,~~0<t\leq T,
\end{align}
where  $A$ is a constant independent of $\rho_0,~T$, and $\theta$.
\end{theorem}
\begin{proof}
Let $\rho_1(t)$ and $\rho_2(t)$ be defined by \eqref{eq:rho12}. Due to \eqref{eq:assumptiona}, we have
\begin{align}\label{eq:Lq1}
\|\rho_1(t)\|_{L^q}\leq \theta t^{-\frac{d}{2}\left(1-\frac{1}{q}\right)},\quad 1\leq q\leq \infty,\quad 0< t\leq T.
\end{align}
Next, we will use three steps to deal with the second term $\rho_2$ and prove \eqref{eq:q}.

\textbf{Step 1.} In this step, we are going to prove that inequality \eqref{eq:q} holds for $1\leq q<\frac{d}{d-1}$ and complete the proof of \eqref{eq:q} for $d=1$. Here, we assume $\frac{d}{d-1}=\infty$ when $d=1$. 

From Young's inequality and \eqref{eq:assumptiona}, we have
\begin{align*}
\|\rho_2(t)\|_{L^q}\leq& C\int_0^t (t-s)^{-\frac{1}{2}-\frac{d}{2}\left(1-\frac{1}{q}\right)}s^{-\frac{1}{2}}\di s\cdot \sup_{0<t\leq T}t^{\frac{1}{2}}\|\rho(t)\|_{L^{\frac{2d}{2d-1}}}^2\nonumber\\
\leq&C\theta^2t^{-\frac{d}{2}\left(1-\frac{1}{q}\right)}\int_0^1 (1-s)^{-\frac{1}{2}-\frac{d}{2}\left(1-\frac{1}{q}\right)}s^{-\frac{1}{2}} \di s.
\end{align*}
For $1\leq q<\frac{d}{d-1}$,  we have 
\[
-\frac{1}{2}-\frac{d}{2}\left(1-\frac{1}{q}\right)> -1,
\]
and hence 
\[
\|\rho_2(t)\|_{L^q}\leq C\theta^2t^{-\frac{d}{2}\left(1-\frac{1}{q}\right)},\quad 1\leq q<\frac{d}{d-1}.
\]
Together with \eqref{eq:Lq1}, we obtain
\begin{align}\label{eq:d1}
\|\rho(t)\|_{L^q(\mathbb{R}^d)}\leq C\theta t^{-\frac{d}{2}(1-\frac{1}{q})}, \quad 1\leq q<\frac{d}{d-1}
\end{align}
for some constant $C$.

When $d=1$, we prove the case  $q=\infty$ as follows.  From Young's inequality, we have
\begin{align*}
\|\rho_2(t)\|_{L^\infty}\leq &\int_0^t\|\nabla G(\cdot,t-s)\|_{L^2}\|\rho(s)\partial_x c(s)\|_{L^2}\di s\\
\leq&C\int_0^t(t-s)^{-\frac{3}{4}}\|\rho(s)\partial_x c(s)\|_{L^2}\di s\leq C\int_0^t(t-s)^{-\frac{3}{4}}\|\rho(s)\|_{L^{4}}^2\di s\\
\leq&C\int_0^t(t-s)^{-\frac{3}{4}} s^{-\frac{3}{4}}\di s\cdot \sup_{0< s\leq t} s^{\frac{3}{4}}\|\rho(s)\|_{L^4}^2\leq C\theta^2 t^{-\frac{1}{2}}.
\end{align*}
Combining \eqref{eq:Lq1}, we obtain \eqref{eq:q} for $d=1$ and $q=\infty$.

\textbf{Step 2.}
In this step, we are going to prove that inequality \eqref{eq:q} holds for $1\leq q<\frac{d}{d-2}$  and complete the proof of \eqref{eq:q} for $d=2$. Here, we assume $\frac{d}{d-2}=\infty$ when $d=2$. 

For any $1\leq q<\frac{d}{d-2}$, there exist $1\leq a,~p<\frac{d}{d-1}$ such that
\[
1+\frac{1}{q}=\frac{1}{a}+\frac{1}{p}.
\]
By Young's convolution inequality, we obtain
\begin{equation*}
\begin{aligned}
\|\rho_2(t)\|_{L^q}\leq \int_0^t\|\nabla G(\cdot,t-s)\ast (\rho(s)\nabla  c(s))\|_{L^q}\di s
\leq C\int_0^t(t-s)^{-\frac{1}{2}-\frac{d}{2}\left(1-\frac{1}{a}\right)}\| {\rho}(s)\nabla c(s)\|_{L^p}\di s.
\end{aligned}
\end{equation*}
Because $1\leq p<\frac{d}{d-1}$, we could find $p< p_1<\frac{d}{d-1}$ and $d<p_2$ such that $\frac{1}{p}=\frac{1}{p_1}+\frac{1}{p_2}$, and then H\"older's inequality implies
\[
\|\rho_2(t)\|_{L^q}\leq C\int_0^t(t-s)^{-\frac{1}{2}-\frac{d}{2}\left(1-\frac{1}{a}\right)}\|\rho(s)\|_{L^{p_1}} \|\nabla c(s)\|_{L^{p_2}}\di s.
\]
Using Hardy-Littlewood-Sobolev inequality with $1+\frac{1}{p_2}=\frac{1}{d}+\frac{1}{p_3}$,  we obtain
\begin{equation*}
\begin{aligned}
\|\rho_2(t)\|_{L^q}\leq C\int_0^t(t-s)^{-\frac{1}{2}-\frac{d}{2}\left(1-\frac{1}{a}\right)}\|\rho(s)\|_{L^{p_1}} \|\rho(s)\|_{L^{p_3}}\di s,
\end{aligned}
\end{equation*}
and the following relation holds
\[
1+\frac{1}{p}=\frac{1}{p_1}+\frac{1}{p_3}+\frac{1}{d}.
\]   
For $1\leq p<\frac{d}{d-1}$, there exists $1<p_1,~~p_3<\frac{d}{d-1}$ satisfying the above relation.  In this case, we apply Step 1 to obtain
\begin{equation}\label{eq:step2}
\begin{aligned}
\|\rho_2(t)\|_{L^q}
\leq &C\theta^2\int_0^t(t-s)^{-\frac{1}{2}-\frac{d}{2}\left(1-\frac{1}{a}\right)}s^{-\frac{1}{2}-\frac{d}{2}\left(1-\frac{1}{p}\right)}\di s\\
\leq& C\theta^2t^{{-\frac{d}{2}\left(1-\frac{1}{q}\right)}}\int_0^1(1-s)^{-\frac{1}{2}-\frac{d}{2}\left(1-\frac{1}{a}\right)}s^{-\frac{1}{2}-\frac{d}{2}\left(1-\frac{1}{p}\right)}\di s.
\end{aligned}
\end{equation}
For $1\leq a,~p<\frac{d}{d-1}$, we have 
\begin{align*}
-\frac{1}{2}-\frac{d}{2}\left(1-\frac{1}{a}\right)>-1
\quad\textrm{and}\quad
-\frac{1}{2}-\frac{d}{2}\left(1-\frac{1}{p}\right)>-1.
\end{align*}
As a consequence, when $1\leq q<\frac{d}{d-2}$, the last integration of \eqref{eq:step2} is finite, and combining \eqref{eq:Lq1}, we obtain
\begin{align*}
\|\rho(t)\|_{L^q(\mathbb{R}^d)}\leq C\theta t^{-\frac{d}{2}(1-\frac{1}{q})},\quad 1\leq q<\frac{d}{d-2}
\end{align*}
for some constant $C$.

When $d=2$, we prove the case  $q=\infty$ as follows. 
Split the integral in $\rho_2(t)$ into two parts:
\begin{equation*}
\begin{aligned}
\|\rho_2(t)\|_{L^\infty}\leq& \int_0^t\|\nabla G(\cdot,t-s)\ast (\rho(s)\nabla  c(s))\|_{L^\infty}\di s\\
=& \int_0^{t/2}\|\nabla G(\cdot,t-s)\ast (\rho(s)\nabla  c(s))\|_{L^\infty}\di s+ \int_{t/2}^t\|\nabla G(\cdot,t-s)\ast (\rho(s)\nabla  c(s))\|_{L^\infty}\di s\\
=&:I_1+I_2.
\end{aligned}
\end{equation*}
For the first term, we have
\begin{align*}
I_1\leq &\int_0^{t/2}\|\nabla G(\cdot,t-s)\|_{L^\infty}\| \rho(s)\nabla  c(s)\|_{L^1}\di s\\
\leq&C\int_0^{t/2}(t-s)^{-\frac{3}{2}}\|\rho(s)\nabla  c(s)\|_{L^1}\di s\leq C\left(\frac{t}{2}\right)^{-\frac{3}{2}}\int_0^{t/2} \|\rho(s)\|^2_{L^{\frac{4}{3}}}\di s\\
\leq&C\theta^2\left(\frac{t}{2}\right)^{-\frac{3}{2}}\int_0^{t/2} s^{-\frac{1}{2}}\di s\leq C\theta^2 t^{-1}.
\end{align*}
For $I_2$, we have
\begin{align*}
I_2\leq &\int_{t/2}^t\|\nabla G(\cdot,t-s)\|_{L^{\frac{4}{3}}} \|\rho(s)\nabla  c(s)\|_{L^{4}}\di s\leq C\int_{t/2}^t (t-s)^{-\frac{3}{4}} \|\rho(s)\|_{L^8}\|\nabla  c(s)\|_{L^{8}}\di s\\
\leq&C\int_{t/2}^t (t-s)^{-\frac{3}{4}} \|\rho(s)\|_{L^8}\|\rho(s)\|_{L^{\frac{8}{5}}}\di s \leq C\theta^2\int_{t/2}^t (t-s)^{-\frac{3}{4}} s^{-\frac{5}{4}}\di s\\
\leq& C\theta^2 \left(\frac{t}{2}\right)^{-\frac{5}{4}}\int_{t/2}^t (t-s)^{-\frac{3}{4}} \di s\leq C\theta^2 t^{-1}.
\end{align*}
Combining the above two estimates for $I_1$, $I_2$ and the estimate \eqref{eq:Lq1}, we obtain \eqref{eq:q} for $d=2$ and $q=\infty$.
	
\textbf{Step 3.}
As we see from Step 2, we could use the results in previous steps to enhance the regularity of the mild solutions. However, the argument for $d\geq 3$ is a little different from Step 1 and Step 2. To be more clear, we provide one more step here. 

In this step, we are going to prove that inequality \eqref{eq:q} holds for $1\leq q<\frac{d}{d-3}$   and complete the proof of \eqref{eq:q} for $d=3$. Here, we assume $\frac{d}{d-3}=\infty$ for $d=3$.

When $1\leq q<\frac{d}{d-3}$, there exist 
\[
1\leq a_1,~p_2<\frac{d}{d-1}~\textrm{ and }~ 1\leq p_1,~a_2<\frac{d}{d-2}
\]
such that $1+\frac{1}{q}=\frac{1}{a_i}+\frac{1}{p_i}$, $i=1,2$. Different from Step 1 and Step 2, we need to separate the integral in $\rho_2(t)$ into two parts. By Young's inequality, we obtain
\begin{equation*}
\begin{aligned}
&\|\rho_2(t)\|_{L^q}\leq \int_0^t\|\nabla G(\cdot,t-s)\ast (\rho(s)\nabla  c(s))\|_{L^q}\di s\\
\leq& C\left[\int_{\frac{1}{2}t}^t(t-s)^{-\frac{1}{2}-\frac{d}{2}\left(1-\frac{1}{a_1}\right)}\| {\rho}(s)\nabla c(s)\|_{L^{p_1}}\di s+\int_0^{\frac{1}{2}t}(t-s)^{-\frac{1}{2}-\frac{d}{2}\left(1-\frac{1}{a_2}\right)}\| {\rho}(s)\nabla c(s)\|_{L^{p_2}}\di s\right]\\
=&:I_1+I_2.
\end{aligned}
\end{equation*}
The main idea is that we only need to choose $1\leq a_1,~p_2<\frac{d}{d-1}$ to make the above two terms integrable, and hence by adjusting $p_1$ and $a_2$, we can adjust the range of $q$.
We will only handle the first term $I_1$, and the proof for $I_2$ is similar.

From H\"older's inequality with $\frac{1}{p_1}=\frac{1}{p_{11}}+\frac{1}{p_{12}}$ and Hardy-Littlewood-Sobolev inequality with $p_{12}>d$ and $1+\frac{1}{p_{12}}=\frac{1}{d}+\frac{1}{p_{13}}$, we have
\begin{align*}
I_1\leq C\int_{\frac{1}{2}t}^t(t-s)^{-\frac{1}{2}-\frac{d}{2}\left(1-\frac{1}{a_1}\right)}\|\rho(s)\|_{L^{p_{11}}}\|\rho(s)\|_{L^{p_{13}}}\di s,
\end{align*}
and the following relation holds
\begin{align*}
1+\frac{1}{p_1}=\frac{1}{p_{11}}+\frac{1}{p_{13}}+\frac{1}{d}.
\end{align*}
To apply the Hardy-Littlewood-Sobolev inequality, we need to choose $p_{13}<\frac{d}{d-1}$. Recall that  $1\leq p_1<\frac{d}{d-2}$. Hence, there  exist $p_1< p_{11}<\frac{d}{d-2}$ and $1\leq p_{13}<\frac{d}{d-1}$ satisfying the above relation.
Therefore, we apply the results in Step 1 and Step 2 to obtain
\begin{align*}
I_1\leq &C\theta^2 \int_{\frac{1}{2}t}^t(t-s)^{-\frac{1}{2}-\frac{d}{2}\left(1-\frac{1}{a_1}\right)}s^{-\frac{1}{2}-\frac{d}{2}(1-\frac{1}{p_1})}\di s\\
=&C\theta^2t^{-\frac{d}{2}(1-\frac{1}{q})} \int_{\frac{1}{2}}^1(1-s)^{-\frac{1}{2}-\frac{d}{2}\left(1-\frac{1}{a_1}\right)}s^{-\frac{1}{2}-\frac{d}{2}(1-\frac{1}{p_1})}\di s
\end{align*}
Because $1\leq a_1<\frac{d}{d-1}$, we have $-\frac{1}{2}-\frac{d}{2}\left(1-\frac{1}{a_1}\right)>-1$ and the above integration is finite. Therefore, $I_1\leq C\theta^2t^{-\frac{d}{2}(1-\frac{1}{q})}$. 

By similar method, we also have  $I_2\leq C\theta^2t^{-\frac{d}{2}(1-\frac{1}{q})}$. Then, combining \eqref{eq:Lq1}, we finish the prove of \eqref{eq:q} for $1\leq q< \frac{d}{d-3}$. When $d=3$, the proof of \eqref{eq:q} for $q=\infty$ is the same as the proof for $d=1$ and $d=2$. This complete the proof of Step 3.

As we  can see from the above steps, we only need to repeat the above process again for $d-3$ times to obtain \eqref{eq:q} for $d>3$.
\end{proof}
\begin{remark} 
Notice that when $\|\rho_0\|_{L^1}$ is small, the constant $A\theta$ in the hypercontractivity estimate \eqref{eq:q} is independent of time $T$.
	
For any initial datum $\rho_0\in L^1(\mathbb{R}^d)$, without the assumption \eqref{eq:assumptiona}, we can also obtain the following estimate by the same method:
\begin{align}\label{eq:qq}
\|\rho(t)\|_{L^q(\mathbb{R}^d)}\leq C t^{-\frac{d}{2}(1-\frac{1}{q})},\quad  1\leq q\leq\infty,~~0<t\leq T.
\end{align}
Here, the constant $C$ depends on $\rho_0,~\theta$, and $T$. This kind of hypercontractive estimates can also be obtained by contraction arguments even for measure initial data; see \cite{bedrossian2014existence}. Moreover, under the assumption for initial data that $\int_{\mathbb{R}^2}\rho_0\log\rho_0\di x<\infty,~~\int_{\mathbb{R}^2}|x|^2\rho_0\di x<\infty$, and $\|\rho_0\|_{L^1}<8\pi$,  the following refined hypercontractivity can be obtained (see \cite[Theorem 1.1]{liu2016refined})
\[
t^{1-\frac{1}{q}}|\log t|^{(1-\e)/q}\|\rho(t)\|_{L^q}\leq C
\] 
for  some constant $C=C(T;C_T;q;\e)$ and any $q>0$, $t\in(0,1]$, $0<\e<1$.

\end{remark}

Next, we are going to improve the regularity and show the decay estimates of the $L^q$ ($1\leq q\leq\infty$) norm of the space-time derivatives of mild solutions to the mPKS equations  \eqref{eq:KS}. We have
\begin{theorem}\label{thm:regularity}
Let $0\leq \rho_0\in L^{1}(\mathbb{R}^d)$ and $T>0$ satisfy \eqref{eq:assumptiona} for some  constant $\theta$. Let $\beta\in\mathbb{N}^d$ be a multi-index and $k\in\mathbb{N}$ be a nonnegative integer. Then,  there exists a constants $C>0$ depending on $\theta$, $\beta$ and $k$ such that the mild solution $\rho(t)$ obtained by Theorem \ref{thm:contraction} satisfies
\begin{align}\label{eq:regularity}
\left\|D_x^\beta\partial_t^k\rho(t)\right\|_{L^q(\mathbb{R}^d)}\leq Ct^{-\frac{|\beta|}{2}-k-\frac{d}{2}(1-\frac{1}{q})},\quad  1\leq q\leq\infty.
\end{align}
Moreover, $\rho(x,t)\geq 0$ for any $t\in [0,T]$ and $x\in\mathbb{R}^d$.
\end{theorem}
When $k=0$, inequalities similar to \eqref{eq:regularity} were obtained for other equations such as the Navier-Stokes \cite{dong2009optimal,giga2002regularizing,sawada2006on,sawada2005analyticity} and the quasi-geostrophic equation  \cite{dong2008spatial}. When $k\neq0$, similar space-time regularity estimate can also obtained; see \cite[Inequality (48)]{dong2008spatial} for the quasi-geostrophic equation. 
When $d=1$, the instantaneous space-time regularity of mild solutions to the mPKS equations \eqref{eq:KS} with fractional Laplacian was proved by a fractional bootstrapping method; see \cite[Theorem 3.2]{gao2020global}.
Here, we will omit the proof for $k=0$ and directly use it to show \eqref{eq:regularity} with $k\neq 0$.
\begin{proof}
\textbf{Step 1.} (Regularity)
Assume we have already obtained inequality \eqref{eq:regularity} for $k=0$ and any $\beta\in\mathbb{N}^d$, i.e.,
\begin{align}\label{eq:regularity1}
\left\|D_x^\beta\rho(t)\right\|_{L^q(\mathbb{R}^d)}\leq Ct^{-\frac{|\beta|}{2}-\frac{d}{2}(1-\frac{1}{q})},\quad  1\leq q\leq\infty.
\end{align}
We use induction for $k$ to prove  \eqref{eq:regularity}. Due to \eqref{eq:regularity1}, inequality \eqref{eq:regularity} holds for any $\beta\in\mathbb{N}^d$ and $k=0$. Assume \eqref{eq:regularity} holds for any $\beta\in\mathbb{N}^d$ and $k< L$ for some positive integer $L$. We are going to prove that inequality \eqref{eq:regularity} also holds for any $\beta\in\mathbb{N}^d$ and  $k=L$.
	
Direct calculation shows that
\begin{equation}\label{eq:spacetimeL}
\begin{aligned}
\|D_x^\beta\partial_t^L\rho(t)\|_{L^q}\leq& \|D_x^\beta\partial_t^L (G(\cdot,t)\ast\rho_0)\|_{L^q}+ \left\|D_x^\beta\partial_t^L \int_0^t\nabla G(\cdot,t-s)\ast\left[\rho(s)\nabla c(s)\right]\di s \right\|_{L^q}\\
\leq& \left\|D_x^\beta\partial_t^L (G(\cdot,t)\ast\rho_0)\right\|_{L^q}+\sum_{j=0}^{L-1}\left\|D_x^\beta\Delta^{L-1-j}\nabla\cdot\partial_t^{j}[\rho(t)\nabla c(t)]\right\|_{L^q}\\
&+\left\|  \int_0^tD_x^\beta\Delta^L\nabla G(\cdot,t-s)\ast\left[\rho(s)\nabla c(s)\right]\di s \right\|_{L^q}=:I_1+I_2+I_3.
\end{aligned}
\end{equation}
Due to \eqref{eq:estimate7}, we have
\begin{align}\label{eq:I1}
I_1\leq Ct^{-\frac{|\beta|}{2}-L-\frac{d}{2}(1-\frac{1}{q})}.
\end{align}
For $I_3$, we use the relation:
\[
-\int_0^tD_x^\beta\Delta^L\nabla G(\cdot,t-s)\ast\left[\rho(s)\nabla c(s)\right]\di s =D_x^\beta\Delta^L\rho(t)-D_x^\beta\Delta^L G(\cdot,t)\ast\rho_0.
\]
From  \eqref{eq:regularity1} and \eqref{eq:estimate3}, we obtain
\begin{equation}\label{eq:spacetimeI13}
I_3\leq Ct^{-\frac{|\beta|}{2}-L-\frac{d}{2}(1-\frac{1}{q})},\quad 1\leq q<\infty, \quad 0<t\leq T.
\end{equation}
Next, we are going to estimate $I_2.$ Direct calculation shows
\begin{equation*}
\begin{aligned}
I_2=&\sum_{j=0}^{L-1}\left\|D_x^\beta\Delta^{L-1-j}\nabla\cdot\partial_t^{j}[\rho(t)\nabla c(t)]\right\|_{L^q}\\
\leq&\sum_{j=0}^{L-1}\sum_{i=0}^{j}\binom{j}{i}\sum_{\substack{|\tilde{\beta}|=|\beta|+1\\+2(L-1-j)}}\sum_{\gamma\leq\tilde{\beta}}\binom{\tilde{\beta}}{\gamma}\left\|D_x^\gamma\partial_t^i\rho(t)D_x^{\tilde{\beta}-\gamma}\partial_t^{j-i}\nabla c(t)\right\|_{L^q}.
\end{aligned}
\end{equation*}
By H\"older's inequality for
\[
\frac{1}{q}=\frac{1}{q_1}+\frac{1}{q_2},
\] 
we have
\begin{align*}
I_2\leq C\|D_x^\gamma\partial_t^i\rho(t)\|_{L^{q_1}}\|D_x^{\tilde{\beta}-\gamma}\partial_t^{j-i}\nabla c(t)\|_{L^{q_2}}.
\end{align*}
Using Hardy-Littlewood-Sobolev inequality  for
\[
1+\frac{1}{q}-\frac{1}{d}=\frac{1}{q_1}+\frac{1}{q'_2},
\]
we have
\begin{equation}\label{eq:spacetimeI21}
\begin{aligned}
I_2\leq &C\|D_x^\gamma\partial_t^i\rho(t)\|_{L^{q_1}}\|D_x^{\tilde{\beta}-\gamma}\partial_t^{j-i}\rho(t)\|_{L^{q'_2}}\\
\leq&C\left(t^{-\frac{|\gamma|}{2}-i-\frac{d}{2}(1-\frac{1}{q_1})}t^{-\frac{|\tilde{\beta}-\gamma|}{2}-(j-i)-\frac{d}{2}(1-\frac{1}{q'_2})}\right)\leq Ct^{-\frac{|\beta|}{2}-L-\frac{d}{2}(1-\frac{1}{q})}.
\end{aligned}
\end{equation}
Combining \eqref{eq:spacetimeL}, \eqref{eq:I1}, \eqref{eq:spacetimeI13}, and \eqref{eq:spacetimeI21} yields \eqref{eq:regularity}.

\textbf{Step 2.} (nonnegativity) The proof is similar to the proof of \cite[Theorem 3.2]{gao2020global} for one dimensional case, and we will skip it here.
	
\end{proof}

\section{Space-time analyticity}\label{sec:analytic}
In this section, we are going to prove the main theorem, i.e., Theorem \ref{thm1}. To this end, let us first give some useful lemmas. 
\begin{lemma}[Lemma 2.1 in \cite{kahane1969}]\label{lmm:multisequences}
Let $\kappa\in\mathbb{N}^d$ be a multi-index. If either $\delta$ or $\epsilon<-\frac{1}{2}$, then
\begin{align}\label{multisequence}
\sum_{\beta+\gamma=\kappa}\frac{\kappa!}{\beta!\gamma!}\left|\beta\right|^{|\beta|+\delta} \left|\gamma\right|^{|\gamma|+\epsilon}\leq \lambda \left|\kappa\right|^{|\kappa|+max\{\delta, \epsilon\}},
\end{align}
where $\lambda$ depends on $\delta$ and $\epsilon$. Here, we used $0^p=1$ for any $p\in\mathbb{R}$.
\end{lemma}
\begin{lemma}[Lemma 3.3 in \cite{donghongjie2020jfa}]\label{eq:donglemm} 
Let $f$ and $g$ be two smooth function on $\mathbb{R}$, for any integer $k\geq 1$, we have
\begin{multline}\label{eq:timeestimate}
\partial_t^k\left(t^kf(t)g(t)\right)=\sum_{j=0}^k\binom{k}{j}\partial_t^j\left(t^jf(t)\right)\partial_t^{k-j}\left(t^{k-j}g(t)\right)\\
-k\sum_{j=0}^{k-1}\binom{k-1}{j}\partial_t^j\left(t^jf(t)\right)\partial_t^{k-1-j}\left(t^{k-1-j}g(t)\right).
\end{multline}
\end{lemma}
We have the following a priori estimates for the space-time derivatives of mild solutions to the mPKS equations \eqref{eq:KS}.
\begin{proposition}\label{pro:spacetime}
Let  $\beta\in\mathbb{N}^d$ be a multi-index and $L\geq 2$, $k$ be nonnegative integers.  For $0< |\beta|+k\leq L-1$ and $1<q<\infty$, if there exist some constants $0<\delta<1$, $M>0$ independent of $|\beta|$ and $k$ such that 
\begin{align}\label{fineestimate}
\left\|D_x^\beta\partial_t^k\left(t^k\rho(t)\right)\right\|_{L^q}\leq M^{|\beta|+k-\delta}\left(|\beta|+k\right)^{|\beta|+k-1}t^{-\frac{|\beta|}{2}-\frac{d}{2}\left(1-\frac{1}{q}\right)},~~0< |\beta|+k\leq L-1,
\end{align}
then we have 
\begin{align}\label{fineestimate1}
\left\|D_x^\beta\partial_t^k\left(t^k\rho(t)\nabla c(t)\right)\right\|_{L^p}\leq N(1+M^\delta)M^{|\beta|+k-2\delta}\left(|\beta|+k\right)^{|\beta|+k-1}t^{-\frac{|\beta|+1}{2}-\frac{d}{2}\left(1-\frac{1}{p}\right)}
\end{align}
for any $1\leq p<\infty$ and $0<|\beta|+k\leq L-1$, and  we also have
\begin{multline}\label{fineestimate2}
\left\|D_x^\beta\partial_t^k\left(t^k\rho(t)\nabla c(t)\right)\right\|_{L^p}\leq NM^{|\beta|+k-2\delta}\left(|\beta|+k\right)^{|\beta|+k-1}t^{-\frac{|\beta|+1}{2}-\frac{d}{2}\left(1-\frac{1}{p}\right)}\\
+\left\|D_x^\beta\partial_t^k\left(t^k\rho(t)\right)\nabla c(t)\right\|_{L^p}+\left\|\rho(t)D_x^\beta\partial_t^k\left(t^k\nabla c(t)\right)\right\|_{L^p}
\end{multline}
for  any $1\leq p<\infty$  and $|\beta|+k=L$. Here, $N>0$ is a constant  independent of $M$, $\beta$ and $k$.
\end{proposition}
\begin{proof}
From the identity \eqref{eq:timeestimate}, the following inequality holds
\begin{equation}\label{eq:timeestimatenorm}
\begin{aligned}
\left\|D_x^\beta\partial_t^k\left[t^k\rho(t)\nabla c(t)\right]\right\|_{L^p}\leq\sum_{j=0}^k\binom{k}{j}\left\|D_x^\beta\left[\partial_t^j\left(t^j\rho(t)\right)\partial_t^{k-j}\left(t^{k-j}\nabla c(t)\right)\right]\right\|_{L^p}\\
+k\sum_{j=0}^{k-1}\binom{k-1}{j}\left\|D_x^\beta\left[\partial_t^j\left(t^j\rho(t)\right)\partial_t^{k-1-j}\left(t^{k-1-j}\nabla c(t)\right)\right]\right\|_{L^p}\\
\leq\sum_{0\leq\gamma\leq\beta}\binom{\beta}{\gamma}\sum_{j=0}^k\binom{k}{j}\left\| D_x^{\gamma}\partial_t^j\left(t^j\rho(t)\right)D_x^{\beta-\gamma}\partial_t^{k-j}\left(t^{k-j}\nabla c(t)\right)\right\|_{L^p}\\
+k\sum_{0\leq\gamma\leq\beta}\binom{\beta}{\gamma}\sum_{j=0}^{k-1}\binom{k-1}{j}\left\|D_x^\gamma\partial_t^j\left(t^j\rho(t)\right)D_x^{\beta-\gamma}\partial_t^{k-1-j}\left(t^{k-1-j}\nabla c(t)\right)\right\|_{L^p}.
\end{aligned}
\end{equation}
	
\textbf{Step 1:} Assume $0< |\beta|+k\leq L-1$ and we are going to prove \eqref{fineestimate1} in this step.  
From the assumption \eqref{fineestimate}, H\"oder's inequality for $\frac{1}{p}=\frac{1}{p_1}+\frac{1}{p_2},~~d<p_2<\infty$ and Hardy-Littlewood-Sobolev inequality with $1+\frac{1}{p_2}=\frac{1}{d}+\frac{1}{p_3}$, we obtain from \eqref{eq:timeestimatenorm} that
\begin{equation*}
\begin{aligned}
&\left\|D_x^\beta\partial_t^k\left[t^k\rho(t)\nabla c(t)\right]\right\|_{L^p}\\
\leq&C\sum_{\substack{|\gamma|+j\neq 0,\\|\gamma|+j\neq|\beta|+k}}\binom{k}{j}\binom{\beta}{\gamma}\left\|D_x^\gamma\partial_t^j\left(t^j\rho(t)\right)\right\|_{L^{p_1}}\left\|D_x^{\beta-\gamma}\partial_t^{k-j}\left(t^{k-j}\rho(t)\right)\right\|_{L^{p_3}}\\
&+Ck\sum_{\substack{|\gamma|+j\neq 0,\\|\gamma|+j\neq|\beta|+k-1}}\binom{k-1}{j} \binom{\beta}{\gamma}\left\|D_x^\gamma\partial_t^j\left(t^j\rho(t)\right)\right\|_{L^{p_1}}\left\|D_x^{\beta-\gamma}\partial_t^{k-1-j}\left(t^{k-1-j}\rho(t)\right)\right\|_{L^{p_3}}\\
&+C\|\rho(t)\|_{L^{p_1}}\left\|D_x^{\beta}\partial_t^{k}\left(t^{k}\rho(t)\right)\right\|_{L^{p_3}}+C\|D_x^\beta\partial_t^k\left(t^k\rho(t)\right)\|_{L^{p_1}}\left\|\rho(t)\right\|_{L^{p_3}}\\
&+Ck\|\rho(t)\|_{L^{p_1}}\left\|D_x^{\beta}\partial_t^{k-1}\left(t^{k-1}\rho(t)\right)\right\|_{L^{p_3}}+Ck\|D_x^\beta\partial_t^{k-1}\left(t^{k-1}\rho(t)\right)\|_{L^{p_1}}\left\|\rho(t)\right\|_{L^{p_3}}\\
\leq&CM^{|\beta|+k-2\delta}\sum_{\substack{|\gamma|+j\neq 0,\\|\gamma|+j\neq|\beta|+k}}\binom{k}{j}\binom{\beta}{\gamma}\times\left(|\gamma|+j\right)^{|\gamma|+j-1}\left[|\beta-\gamma|+(k-j)\right]^{|\beta-\gamma|+(k-j)-1}t^{-\frac{|\beta|+1}{2}-\frac{d}{2}\left(1-\frac{1}{p}\right)}\\
&+CkM^{|\beta|+k-1-2\delta}k\sum_{\substack{|\gamma|+j\neq 0,\\|\gamma|+j\neq|\beta|+k-1}}\binom{k-1}{j}\binom{\beta}{\gamma}\\
&\qquad \times\left(|\gamma|+j\right)^{|\gamma|+j-1}\left[|\beta-\gamma|+(k-1-j)\right]^{|\beta-\gamma|+(k-1-j)-1}t^{-\frac{|\beta|+1}{2}-\frac{d}{2}\left(1-\frac{1}{p}\right)}\\
&+CM^{|\beta|+k-\delta}(|\beta|+k)^{|\beta|+k-1}t^{-\frac{|\beta|+1}{2}-\frac{d}{2}\left(1-\frac{1}{p}\right)}+CkM^{|\beta|+k-1-\delta}(|\beta|+k-1)^{|\beta|+k-2}t^{-\frac{|\beta|+1}{2}-\frac{d}{2}\left(1-\frac{1}{p}\right)}
\end{aligned}
\end{equation*}
for some constant $C$. Consider $\bar\beta=(\beta_1,\cdots,\beta_d,k)\in\mathbb{N}^{d+1}~\textrm{and}~ \tilde\beta=(\beta_1,\cdots,\beta_d,k-1)\in\mathbb{N}^{d+1}$. Hence, $|\bar\beta|=|\beta|+k$ and $|\tilde\beta|=|\beta|+k-1$, and combining Lemma \ref{lmm:multisequences} gives
\begin{small}
\begin{align*}
&\sum_{\substack{|\gamma|+j\neq 0,\\|\gamma|+j\neq|\beta|+k}}\binom{k}{j}\binom{\beta}{\gamma}\left(|\gamma|+j\right)^{|\gamma|+j-1}\left[|\beta-\gamma|+(k-j)\right]^{|\beta-\gamma|+(k-j)-1}\\
\leq &\sum_{j=0}^k\sum_{0\leq\gamma\leq\beta}\frac{\beta_1!\cdots\beta_d!k!}{(\beta_1-\gamma_1)!\gamma_1!\cdots(\beta_d-\gamma_d)!\gamma_d!(k-j)!j!}\left(|\gamma|+j\right)^{|\gamma|+j-1}\left[|\beta-\gamma|+(k-j)\right]^{|\beta-\gamma|+(k-j)-1}\\
=&\sum_{0\leq\bar\gamma\leq\bar\beta}\binom{\bar\beta}{\bar\gamma}\left|\bar\gamma\right|^{|\bar\gamma|-1}\left|\bar\beta-\bar\gamma\right|^{|\bar\beta-\bar\gamma|-1}\leq\lambda|\bar\beta|^{|\bar\beta|-1},
\end{align*}
\end{small}
and
\begin{multline*}
\sum_{\substack{|\gamma|+j\neq 0,\\|\gamma|+j\neq|\beta|+k-1}}\binom{k-1}{j}\binom{\beta}{\gamma}\left(|\gamma|+j\right)^{|\gamma|+k-1}\left[|\beta-\gamma|+(k-1-j)\right]^{|\beta-\gamma|+(k-1-j)-1}\\
\leq \sum_{0\leq\tilde\gamma\leq\tilde\beta}\binom{\tilde\beta}{\tilde\gamma}\left|\tilde\gamma\right|^{|\tilde\gamma|-1}\left|\tilde\beta-\tilde\gamma\right|^{|\tilde\beta-\tilde\gamma|-1}\leq\lambda|\tilde\beta|^{|\tilde\beta|-1}.
\end{multline*}
Combining all the above inequalities, we obtain
\begin{align*}
&\left\|D_x^\beta\partial_t^k\left[t^k\rho(t)\nabla c(t)\right]\right\|_{L^p}\leq C\lambda M^{|\beta|+k-2\delta}(|\beta|+k)^{|\beta|+k-1}t^{-\frac{|\beta|+1}{2}-\frac{d}{2}\left(1-\frac{1}{p}\right)}\\
&+C\lambda k M^{|\beta|+k-1-2\delta}k(|\beta|+k-1)^{|\beta|+k-2}t^{-\frac{|\beta|+1}{2}-\frac{d}{2}\left(1-\frac{1}{p}\right)}\\
&+CM^{|\beta|+k-\delta}(|\beta|+k)^{|\beta|+k-1}t^{-\frac{|\beta|+1}{2}-\frac{d}{2}\left(1-\frac{1}{p}\right)}+CkM^{|\beta|+k-1-\delta}(|\beta|+k-1)^{|\beta|+k-2}t^{-\frac{|\beta|+1}{2}-\frac{d}{2}\left(1-\frac{1}{p}\right)}\\
&\leq(2C\lambda+C+CM^\delta)M^{|\beta|+k-2\delta}\left(|\beta|+k\right)^{|\beta|+k-1}t^{-\frac{|\beta|+1}{2}-\frac{d}{2}\left(1-\frac{1}{p}\right)}.
\end{align*}
Let $N=2\lambda C+C$ and we obtain \eqref{fineestimate1}.

\textbf{Step 2:} Consider the case: $|\beta|+k=L$. Notice that the terms for $|\gamma|=j=0$ and $\gamma=\beta,~~j=k$ in the summation $I_1$ of \eqref{eq:timeestimatenorm} are 
\[
\left\|D_x^\beta\partial_t^k\left(t^k\rho(t)\right)\nabla c(t)\right\|_{L^p}+\left\|\rho(t)D_x^\beta\partial_t^k\left(t^k\nabla c(t)\right)\right\|_{L^p}.
\]
Take these two terms out and by the same arguments as in Step 1 for the rest terms in \eqref{eq:timeestimatenorm}, we obtain \eqref{fineestimate2}.
\end{proof}

Now, we are ready to state and prove one of the main results in this paper, which includes Theorem \ref{thm1} (i):
\begin{theorem}\label{thm:analytic}
Let $0\leq \rho_0\in L^{1}(\mathbb{R}^d)$ and $T>0$ satisfy \eqref{eq:assumptiona} for some small constant $\theta$. Let $\rho(t)$ be the mild solution to the mPKS equations  \eqref{eq:KS} in $[0,T]$, $\beta\in\mathbb{N}^d$ be a multi-index and $k\in\mathbb{N}$ be a nonnegative integer. 
Then there exists a constant  $M$ independent of $\beta$, $k$, and $T$, such that 
\begin{align}\label{eq:analytic}
\left\|D_x^\beta\partial_t^k\rho(t)\right\|_{L^q(\mathbb{R}^d)}\leq M^{|\beta|+k}\left(|\beta|+k\right)^{|\beta|+k}t^{-\frac{|\beta|}{2}-k-\frac{d}{2}(1-\frac{1}{q})}
\end{align}
holds for $1\leq q\leq\infty$, $ t\in(0,T]$, and $|\beta|+k>0$.
Moreover, if $\|\rho_0\|_{L^1}$ is small enough, inequality \eqref{eq:analytic} holds for any $t\in(0,\infty)$.
\end{theorem}
\begin{proof}
Let us introduce some notations for this proof. 
\begin{enumerate}
\item[(i)] For $\beta\in\mathbb{N}^d$ and $1\leq q\leq \infty$, set
\begin{align}\label{eq:mu}
\mu:=\frac{|\beta|}{2}+\frac{d}{2}\left(1-\frac{1}{q}\right).
\end{align}
\item[(ii)] There exists a constant $A_0$ independent of $p$ such that
\begin{equation}\label{eq:AA0}
\left\|\nabla G(\cdot,t-s)\right\|_{L^{p}}\leq A_0(t-s)^{-\frac{1}{2}-\frac{d}{2}(1-\frac{1}{p})},\quad 1\leq p\leq \infty.
\end{equation}
\item[(iii)] In the rest of the prove, we will use $A_1$ to denote the constant generated by Hardy-Littlewood-Sobolev inequality. 
\end{enumerate}
Instead of proving \eqref{eq:analytic} directly, we fix $0<\delta<1$ and use induction to show that there exists some constant $M$ independent of $|\beta|$, $k$, and $T$ such that
\begin{align}\label{eq:Linduction}
\left\|D_x^\beta\partial_t^k\left(t^k\rho(t)\right)\right\|_{L^{q}}\leq M^{|\beta|+k-\delta}(|\beta|+k)^{|\beta|+k-1}t^{-\mu}
\end{align}
for any $1\leq q\leq\infty$, $t\in(0,T]$,  and $|\beta|+k>0$.  We first prove the following claim:

\

\noindent\textbf{Claim:} If inequality \eqref{eq:Linduction} holds for some constant $M$, then inequality \eqref{eq:analytic} holds for $M+1$. 
\begin{proof}[Proof of the Claim]
Assume $k>0$. Notice that for $j=0,1,2,\cdots,k-1$, we have
\begin{equation*}
\begin{aligned}
t^jD_x^\beta\partial_t^k\left(t^{k-j}\rho\right)=&t^j\left[D_x^\beta\sum_{i=0}^k\binom{k}{i}\partial_t^it\partial_t^{k-i}\left(t^{k-j-1}\rho\right)\right]\\
=&t^j\left[D_x^\beta t\partial_t^k\left(t^{k-j-1}\rho\right)+kD_x^\beta\partial_t^{k-1}\left(t^{k-j-1}\rho\right)\right]\\
=&t^{j+1}\left[D_x^\beta\partial_t^k\left(t^{k-j-1}\rho\right)\right]+kt^j\left[D_x^\beta\partial_t^{k-1}\left(t^{k-j-1}\rho\right)\right],
\end{aligned}
\end{equation*}
which implies
\begin{align}\label{eq:kinduc}
t^{j+1}\left[D_x^\beta\partial_t^k\left(t^{k-j-1}\rho\right)\right]=t^jD_x^\beta\partial_t^k\left(t^{k-j}\rho\right)-kt^j\left[D_x^\beta\partial_t^{k-1}\left(t^{k-j-1}\rho\right)\right].
\end{align}
Combining $j=0$ in \eqref{eq:kinduc} and inequality \eqref{eq:Linduction}, we have
\begin{align*}
\sup_{t\in(0,T]}\left\|t^{\mu+1}D_x^\beta\partial_t^k\left(t^{k-1}\rho\right)\right\|_{L^q}\leq&\sup_{t\in(0,T]}\left\|t^{\mu}D_x^\beta\partial_t^k(t^k\rho)\right\|_{L^q}+\sup_{t\in(0,T]}\left\|kt^{\mu}D_x^\beta\partial_t^{k-1}(t^{k-1}\rho)\right\|_{L^q}\\
\leq&M^{|\beta|+k-\delta}\left(1+\frac{1}{M}\right)\left(|\beta|+k\right)^{|\beta|+k-1}.
\end{align*}
Similarly, for any $0<\tilde{k}\leq k$ we have
\begin{align*}
\sup_{t\in(0,T]}\left\|t^{\mu+1}D_x^\beta\partial_t^{\tilde{k}}\left(t^{\tilde{k}-1}\rho\right)\right\|_{L^q}
\leq M^{|\beta|+\tilde{k}-\delta}\left(1+\frac{1}{M}\right)\left(|\beta|+\tilde{k}\right)^{|\beta|+\tilde{k}-1}.
\end{align*}
Repeat the above process for $j=1,2,\cdots,k-1$ and we obtain
\begin{align*}
\sup_{t\in(0,T]}\left\|t^{\mu+k}D_x^\beta\partial_t^k\rho\right\|_{L^q}\leq M^{|\beta|+k-\delta}\left(1+\frac{1}{M}\right)^k\left(|\beta|+k\right)^{|\beta|+k-1},
\end{align*}
which implies that inequality \eqref{eq:analytic} holds for $M+1$.
\end{proof}

In the following, we will focus on the proof of \eqref{eq:Linduction}. The proof will be separated into two parts: Part I for $1< q<\infty$ and Part II for $q=\infty$ and $q=1$.

\

\noindent \textbf{Part I ($1<q<\infty$):}
We will use an induction for $|\beta|+k$ to prove inequality \eqref{eq:Linduction} with $1< q<\infty$.
Start from $|\beta|+k=1$. For $k=0$ and $|\beta|=1$, due to \eqref{eq:regularity}, there exists a constant $M>0$ satisfying \eqref{eq:Linduction}. For $|\beta|=0$ and $k=1$, by \eqref{eq:regularity} again, we have
\begin{align*}
\|\partial_t(t\rho(t))\|_{L^q}\leq\|\rho(t)\|_{L^q}+t\|\partial_t\rho(t)\|_{L^q}	\leq A\theta t^{-\frac{d}{2}\left(1-\frac{1}{q}\right)}+tCt^{-1-\frac{d}{2}\left(1-\frac{1}{q}\right)}\leq Ct^{-\frac{d}{2}\left(1-\frac{1}{q}\right)}.
\end{align*}
Therefore, we can always find some constant $M$ such that \eqref{eq:Linduction} holds for $|\beta|+k=1$.

Assume that there exists a constant $M$  independent of $\beta$ and $k$ such that \eqref{eq:Linduction} holds for $0<|\beta|+k\leq L-1$ for some $L\geq 2$.
We are going to use four steps to prove that \eqref{eq:Linduction} also holds for $|\beta|+k=L$ with the same $M$. The first two steps are for $|\beta|>0$, and the last two steps are for $|\beta|=0$.
\begin{enumerate}
\item [$\bullet$] \textbf{Step 1:} For $|\beta|>0$, $|\beta|+k=L$, and $1< q<\frac{d}{d-1}$, we are going to prove 
\begin{align}\label{eq:Linduction1}
\left\|D_x^\beta\partial_t^k\left(t^k\rho(t)\right)\right\|_{L^{q}}\leq h(M) M^{|\beta|+k-\delta}(|\beta|+k)^{|\beta|+k-1}t^{-\mu},\quad \forall t\in(0,T],
\end{align}
where function $h$ is independent of $q$ and $|\beta|+k$, and satisfies $h(M)\to0$ as $M\to\infty.$
\item [$\bullet$] \textbf{Step 2:} We are going to prove \eqref{eq:Linduction} for  $|\beta|>0$, $|\beta|+k=L$, and $1< q<\infty$.
\item [$\bullet$] \textbf{Step 3:} For $|\beta|=0$, $k=L$, and $1< q<\frac{d}{d-1}$, we are going to prove 
\begin{align}\label{eq:Linduction2}
\left\|\partial_t^k\left(t^k\rho(t)\right)\right\|_{L^{q}}\leq \bar{h}(M) M^{k-\delta}k^{k-1}t^{-\mu},\quad \forall t\in(0,T],
\end{align}  
where function $\bar{h}$ is independent of $q$ and $|\beta|+k$, and satisfies $\bar{h}(M)\to0$ as $M\to\infty.$
\item [$\bullet$] \textbf{Step 4:} We are going to prove \eqref{eq:Linduction}  for  $|\beta|=0$, $k=L,$ and $1< q< \infty$.
\end{enumerate}
Here, for any positive integer $\ell$, we will assume $\frac{d}{d-\ell}=\infty$ for $d=\ell.$ Notice that combining the results in Step 2 and Step 4, we will complete the induction process and finish the proof of Part I for $1< q<\infty$.

\

\noindent\textbf{Step 1 ($|\beta|>0$, $|\beta|+k=L$, $1< q<\frac{d}{d-1}$):}
Direct calculation shows that
\begin{multline}\label{eq:induction}
\left\|D_x^\beta\partial_t^k\left(t^k\rho(t)\right)\right\|_{L^q}\leq \left\|D_x^\beta\partial_t^k\left[t^kG(\cdot,t)\ast\rho_0\right]\right\|_{L^q}\\
	+\left\|D_x^\beta\partial_t^k\left[t^k\int_0^t\nabla G(\cdot,t-s)\ast\left[\rho(s)\nabla c(s)\right]\di s\right]\right\|_{L^q}.
\end{multline}
From inequality \eqref{eq:estimate7} with $f=\rho_0$ and $m=k$, the first term  in \eqref{eq:induction} becomes
\begin{equation*}
\begin{aligned}
\left\|D_x^\beta\partial_t^k\left[t^kG(\cdot,t)\ast\rho_0\right]\right\|_{L^q}\leq M_0^{|\beta|+k}(|\beta|+k)^{|\beta|+k+\frac{d}{2}(1-\frac{1}{q})}\|\rho_0\|_{L^1}t^{-\mu}.
\end{aligned}
\end{equation*}
When $M>M_0$, we have $M^\delta\left(\frac{M_0}{M}\right)^{|\beta|+k}(|\beta|+k)^{1+\frac{d}{2}\left(1-\frac{1}{q}\right)}\to 0~\textrm{ as }~|\beta|+k\to \infty.$
Define
\begin{align}\label{eq:h1M}
h_1(M):=\|\rho_0\|_{L^1}\sup_{|\beta|+k\geq 1}\left[M^\delta\left(\frac{M_0}{M}\right)^{|\beta|+k}(|\beta|+k)^{1+\frac{d}{2}}\right],
\end{align}
and then $h_1$ satisfies $\lim_{M\to\infty}h_1(M)=0$. Moreover, we have
\begin{equation}\label{firstterm}
\begin{aligned}
\left\|D_x^\beta\partial_t^k\left[t^kG(\cdot,t)\ast\rho_0\right]\right\|_{L^q}\leq  h_1(M)M^{|\beta|+k-\delta}(|\beta|+k)^{|\beta|+k-1}t^{-\mu}
\end{aligned}
\end{equation}
for any $|\beta|+k>0$ and $1\leq q\leq \infty$.

Next, we estimate the second term in \eqref{eq:induction}.
By the identity $t^k=\sum_{j=0}^k\binom{k}{j}s^{k-j}(t-s)^{j}$, the second term in \eqref{eq:induction} becomes
\begin{equation*}
\begin{aligned}
&\left\|D_x^\beta\partial_t^k\left[t^k\int_0^t\nabla G(\cdot,t-s)\ast\left[\rho(s)\nabla c(s)\right]\di s\right]\right\|_{L^q}\\
=&\left\|D_x^\beta\sum_{j=0}^k\binom{k}{j}\partial_t^k\int_0^t\int_{\mathbb{R}^d}(t-s)^j\nabla G(x-y,t-s)s^{k-j}(\rho\nabla c)(y,s)\di y\di s\right\|_{L^q}\\
=&\left\|D_x^\beta\sum_{j=0}^k\binom{k}{j}\partial_t^{k-j}\int_0^t\int_{\mathbb{R}^d}\partial_t^j\left[(t-s)^j\nabla G(x-y,t-s)\right]s^{k-j}(\rho\nabla c)(y,s)\di y\di s\right\|_{L^q}.
\end{aligned}
\end{equation*}
Changing of variable gives
\begin{equation}\label{shijianjisuan}
\begin{aligned}
&\left\|D_x^\beta\partial_t^k\left[t^k\int_0^t\nabla G(\cdot,t-s)\ast\left[\rho(s)\nabla c(s)\right]\di s\right]\right\|_{L^q}\\
=&\left\|D_x^\beta\sum_{j=0}^k\binom{k}{j}\partial_t^{k-j}\int_0^t\int_{\mathbb{R}^d}\partial_s^j\left[s^j\nabla G(x-y,s)\right](t-s)^{k-j}(\rho\nabla c)(y,t-s)\di y\di s\right\|_{L^q}\\
=&\left\|D_x^\beta\sum_{j=0}^k\binom{k}{j}\int_0^t\int_{\mathbb{R}^d}\partial_s^j\left[s^j\nabla G(x-y,s)\right]\partial_t^{k-j}\left[(t-s)^{k-j}(\rho\nabla c)(y,t-s)\right]\di y\di s\right\|_{L^q}\\
=&\left\|D_x^\beta\sum_{j=0}^k\binom{k}{j}\int_0^t\left[\partial_t^j((t-s)^j\nabla G(\cdot,t-s))\right]\ast\left[\partial_s^{k-j}(s^{k-j}\rho(s)\nabla c(s))\right]\di s\right\|_{L^q}.
\end{aligned}
\end{equation}
For some $0<\e<1$ to be determined later, we separate the integration into two parts and obtain
\begin{equation}\label{eq:case1}
\begin{aligned}
&\left\|D_x^\beta\partial_t^k\left[t^k\int_0^t\nabla G(\cdot,t-s)\ast\left(\rho(s)\nabla c(s)\right)\di s\right]\right\|_{L^q}\\ \leq&\sum_{j=0}^k\binom{k}{j}\int_{(1-\epsilon)t}^t\left\|\left[\partial_t^j((t-s)^j\nabla G(\cdot,t-s))\right]\ast\left[D_x^\beta\partial_s^{k-j}(s^{k-j}\rho(s)\nabla c(s))\right]\right\|_{L^q}\di s\\
&+\sum_{j=0}^k\binom{k}{j}\int_0^{(1-\epsilon)t}\left\|\left[D_x^\beta\partial_t^j((t-s)^j\nabla G(\cdot,t-s))\right]\ast\left[\partial_s^{k-j}(s^{k-j}\rho(s)\nabla c(s))\right]\right\|_{L^q}\di s\\
=&:I_1+I_2.
\end{aligned}
\end{equation}

\textbf{Estimate of the first term  $I_1$ in \eqref{eq:case1}:} 
From Young's inequality for $1+\frac{1}{q}=\frac{1}{a}+\frac{1}{p}$ and $1\leq a,~p\leq q$, we have
\begin{equation}\label{eq:estI1}
\begin{aligned}
I_1\leq&\sum_{j=1}^k\binom{k}{j}\int_{(1-\epsilon)t}^t\left\|\partial_t^j\left[(t-s)^j\nabla G(\cdot,t-s)\right]\right\|_{L^{a}}\left\|D_x^\beta\partial_s^{k-j}\left[s^{k-j}\rho(s)\nabla c(s)\right]\right\|_{L^{p}}\di s\\
&+\int_{(1-\epsilon)t}^t\left\|\nabla G(\cdot,t-s)\right\|_{L^{a}}\left\|D_x^\beta\partial_s^k\left[s^k\rho(s)\nabla c(s)\right]\right\|_{L^{p}}\di s=:I_{11}+I_{12}.
\end{aligned}
\end{equation}
By \eqref{eq:estimate6},  there holds
\[
\left\|\partial_t^j\left[(t-s)^j\nabla G(\cdot,t-s)\right]\right\|_{L^{a}}\leq M_0^{j}j^{j+\frac{1}{2}+\frac{d}{2}(1-\frac{1}{a})}(t-s)^{-\frac{1}{2}-\frac{d}{2}(1-\frac{1}{a})},\quad 1\leq j\leq k.
\]
From \eqref{fineestimate1}, we have 
\[
\left\|D_x^\beta\partial_s^{k-j}\left[s^{k-j}\rho(s)\nabla c(s)\right]\right\|_{L^{p}}\leq N(1+M^\delta) M^{|\beta|+k-j-2\delta}(|\beta|+k-j)^{|\beta|+k-j-1}s^{-\frac{|\beta|+1}{2}-\frac{d}{2}(1-\frac{1}{p})}
\]
for $1\leq j\leq k$. Constant $N$ comes from \eqref{fineestimate1}.
Combining the above two inequalities gives the following estimate for  the term $I_{11}$ in \eqref{eq:estI1}:
\begin{multline}\label{eq:estimateI11}
I_{11}\leq J_1(\e)\sum_{j=1}^k\binom{k}{j}M_0^{j}j^{j+\frac{1}{2}+\frac{d}{2}(1-\frac{1}{a})}N(1+M^\delta) M^{|\beta|+k-j-2\delta}(|\beta|+k-j)^{|\beta|+k-j-1}t^{-\mu}\\
= N(M^{-\delta}+1)M^{|\beta|+k-\delta}J_1(\e)\sum_{j=1}^k\binom{k}{j}\left(\frac{M_0}{M}\right)^jj^{\frac{3}{2}+\frac{d}{2}(1-\frac{1}{a})}j^{j-1}(|\beta|+k-j)^{|\beta|+k-j-1}t^{-\mu},
\end{multline}
where $J_1(\e)$ is defined by
\begin{align}\label{eq:J1e}
J_1(\epsilon)=\int_{1-\epsilon}^1(1-s)^{-\frac{1}{2}-\frac{d}{2}\left(1-\frac{1}{a}\right)}s^{-\frac{|\beta|+1}{2}-\frac{d}{2}\left(1-\frac{1}{p}\right)}\di s.
\end{align}
Here we choose $1\leq a<\frac{d}{d-1}$, and then
$-\frac{1}{2}-\frac{d}{2}\left(1-\frac{1}{a}\right)>-1.$
For some constant $K$ to be determined, we set
\begin{equation}\label{eq:ebeta}
\e=\e(|\beta|):=(K|\beta|)^{-1}.
\end{equation}
Then, we have
\begin{equation}
J_1(\e)\leq C(K|\beta|)^{-(\frac{1}{2}-\frac{d}{2}(1-\frac{1}{a}))}\left(1+\frac{1}{K|\beta|-1}\right)^{\frac{|\beta|+1}{2}+\frac{d}{2}\left(1-\frac{1}{p}\right)}.
\end{equation}
By the choice of $a$, we have $\frac{1}{2}-\frac{d}{2}(1-\frac{1}{a})>0$ and hence $J_1(\e)\to 0$ as $K|\beta|\to\infty$. Therefore, we fixed $K$ big enough such 
\begin{align}\label{eq:J}
J_1(\e)\leq J:=\frac{1}{4AA_0A_1\theta }
\end{align}
holds for any $|\beta|>0$, where $A$ comes from \eqref{eq:q} and $A_0$ satisfies \eqref{eq:AA0}.
(Here, we choose the above two constants $J$ and $K$ because they will also be used in the following estimate of \eqref{eq:estI12}.)
For any $M>M_0$, notice that $\left(\frac{M_0}{M}\right)^jj^{\frac{3}{2}+\frac{d}{2}(1-\frac{1}{a})}\to 0$ as $j\to\infty$. Set
\begin{align}\label{eq:h2M}
h_2(M):= 2\lambda N J \sup_{j\geq 1}\left[\left(\frac{M_0}{M}\right)^jj^{\frac{3}{2}+\frac{d}{2} }\right]
\end{align}
and then $h_2$ satisfies $\lim_{M\to\infty}h_2(M)=0$, where $\lambda$ is the constant in Lemma \ref{lmm:multisequences}.
Therefore, from \eqref{eq:estimateI11} we obtain
\begin{equation}\label{eq:estI11}
\begin{aligned}
I_{11}\leq h_2(M)M^{|\beta|+k-\delta}(|\beta|+k)^{|\beta|+k-1}t^{-\mu},
\end{aligned}
\end{equation}
where  we used Lemma \ref{lmm:multisequences} in the last step.

For the term $I_{12}$ in \eqref{eq:estI1},  we obtain from  \eqref{fineestimate2} and H\"older's inequality that
\begin{align*}
&\left\|D_x^\beta\partial_s^k\left[s^k\rho(s)\nabla c(s)\right]\right\|_{L^{p}}\leq NM^{|\beta|+k-2\delta}\left(|\beta|+k\right)^{|\beta|+k-1}s^{-\frac{|\beta|+1}{2}-\frac{d}{2}\left(1-\frac{1}{p}\right)}\\
&\qquad \qquad +\left\|D_x^\beta\partial_s^k\left(s^k\rho(s)\right)\nabla c(s)\right\|_{L^{p}}+\left\|\rho(s)D_x^\beta\partial_s^k\left(s^k\nabla c(s)\right)\right\|_{L^{p}}\\
\leq&NM^{|\beta|+k-2\delta}\left(|\beta|+k\right)^{|\beta|+k-1}s^{-\frac{|\beta|+1}{2}-\frac{d}{2}\left(1-\frac{1}{p}\right)}\\
&+\left\|D_x^\beta\partial_s^k(s^k\rho(s))\right\|_{L^q}\|\nabla c( s)\|_{L^{\frac{pq}{q-p}}}+\|\rho(s)\|_{L^{\frac{1}{2-\frac{1}{a}-\frac{1}{d}}}}\left\|D_x^\beta\partial_s^k\left(s^k\nabla c(s)\right)\right\|_{L^{\frac{1}{\frac{1}{q}+\frac{1}{d}-1}}}.
\end{align*}
Here, we used the assumption $1< q< \frac{d}{d-1}$ to estimate $I_{12}$. (This is the only place in Step 1 which needs the assumption $1< q< \frac{d}{d-1}$. Here we could also assume $q=1$ for $d>1$.)
By \eqref{eq:AA0},  we have the following estimate for the term $I_{12}$ in \eqref{eq:estI1}:
\begin{equation}\label{eq:estm}
\begin{aligned}
I_{12}\leq& NA_0M^{|\beta|+k-2\delta}\left(|\beta|+k\right)^{|\beta|+k-1}J_1(\epsilon)t^{-\frac{|\beta|}{2}-\frac{d}{2}\left(1-\frac{1}{q}\right)}\\
&+A_0\int_{(1-\epsilon)t}^t\left\|D_x^\beta \partial_s^k(s^k\rho(s)) \right\|_{L^q}\|\nabla c( s)\|_{L^{\frac{pq}{q-p}}}(t-s)^{-\frac{1}{2}-\frac{d}{2}\left(1-\frac{1}{a}\right)}\di s\\
&+A_0\int_{(1-\epsilon)t}^t\|\rho(s)\|_{L^{\frac{1}{2-\frac{1}{a}-\frac{1}{d}}}}\left\|D_x^\beta\partial_s^k\left(s^k\nabla c(s)\right)\right\|_{L^{\frac{1}{\frac{1}{q}+\frac{1}{d}-1}}}(t-s)^{-\frac{1}{2}-\frac{d}{2}\left(1-\frac{1}{a}\right)}\di s\\
\leq&NA_0JM^{|\beta|+k-2\delta}\left(|\beta|+k\right)^{|\beta|+k-1}t^{-\mu}\\
&+2AA_0A_1\theta\int_{(1-\epsilon)t}^t\left\|D_x^\beta\partial_s^k(s^k\rho(s))\right\|_{L^q}(t-s)^{-\frac{1}{2}-\frac{d}{2}\left(1-\frac{1}{a}\right)}s^{-\frac{d}{2}\left(-1+\frac{1}{a}+\frac{1}{d}\right)}\di s,
\end{aligned}
\end{equation}
where $A_1$ comes from Hardy-Littlewood-Sobolev inequality.
Set
\begin{align}\label{eq:phi}
\phi(t):=\sup_{0<s\leq t}{s^{\mu}\left\|D_x^\beta\partial_s^k\left(s^k\rho(s)\right)\right\|_{L^q}}.
\end{align}
Set
\begin{align}\label{eq:h3M}
h_3(M):=NA_0JM^{-\delta},
\end{align}
and by the choice of $J_1(\epsilon)$ in \eqref{eq:J}, we obtain
\begin{equation}\label{eq:estI12}
\begin{aligned}
I_{12}\leq&h_3(M)M^{|\beta|+k-\delta}\left(|\beta|+k\right)^{|\beta|+k-1}t^{-\mu}+2AA_0A_1\theta J_1(\e)\cdot\phi(t)t^{-\mu}\\
\leq&h_3(M)M^{|\beta|+k-\delta}\left(|\beta|+k\right)^{|\beta|+k-1}t^{-\mu}+\frac{1}{2}\cdot\phi(t)t^{-\mu}.
\end{aligned}
\end{equation}
Combining \eqref{eq:estI1}, \eqref{eq:estI11} and \eqref{eq:estI12} gives
\begin{equation}\label{eq:estimateI1}
\begin{aligned}
I_1\leq [h_2(M)+h_3(M)]M^{|\beta|+k-\delta}\left(|\beta|+k\right)^{|\beta|+k-1}t^{-\mu}+\frac{1}{2}\phi(t)t^{-\mu}.
\end{aligned}
\end{equation}

\textbf{Estimate of the second term $I_2$ in \eqref{eq:case1}:}
Combining Young's convolution inequality for $1+\frac{1}{q}=\frac{1}{a}+\frac{1}{p}$ and $1\leq a,~p\leq q$, \eqref{eq:estimate6} and \eqref{fineestimate1}, we have
\begin{equation}\label{I_2}
\begin{aligned}
I_2\leq&\sum_{j=0}^{k}\binom{k}{j}\int_0^{(1-\epsilon)t}\left\|D_x^\beta\partial_t^j\left[(t-s)^j\nabla G(\cdot,t-s)\right]\right\|_{L^{a}}\left\|\partial_s^{k-j}\left[s^{k-j}\rho(s)\nabla c(s)\right]\right\|_{L^{p}}\di s\\
\leq&\sum_{j=0}^k\binom{k}{j}M_0^{\frac{|\beta|}{2}+j}(|\beta|+j)^{\frac{|\beta|+1}{2}+j+\frac{d}{2}(1-\frac{1}{a})} N(1+M^\delta)M^{k-j-2\delta}(k-j)^{k-j-1}J_2(\epsilon)t^{-\mu},
\end{aligned}
\end{equation}
where 
\begin{align}\label{eq:J2ep}
J_2(\epsilon)=\int_0^{1-\epsilon}(1-s)^{-\frac{|\beta|+1}{2}-\frac{d}{2}\left(1-\frac{1}{a}\right)}s^{-\frac{1}{2}-\frac{d}{2}\left(1-\frac{1}{p}\right)}\di s,
\end{align}
and $N$ comes from inequality \eqref{fineestimate1}.
Here, we choose $1\leq p<\frac{d}{d-1}$, and we have 
$-\frac{1}{2}-\frac{d}{2}\left(1-\frac{1}{p}\right)>-1.$
For $\epsilon=(K|\beta|)^{-1}$, we have $J_2(\epsilon)\leq C(K|\beta|)^{\frac{|\beta|+1}{2}+\frac{d}{2}\left(1-\frac{1}{a}\right)}$
for some constant $C$, which implies
\begin{multline*}
I_2\leq CN(M^{-\delta}+1)M^{|\beta|+k-\delta}\sum_{j=0}^k\binom{k}{j}\frac{M_0^{\frac{|\beta|}{2}+j}K^{\frac{|\beta|+1}{2}+\frac{d}{2}(1-\frac{1}{a})}}{M^{|\beta|+j}}(|\beta|+j)^{2+d(1-\frac{1}{a})}\\
\times(|\beta|+j)^{|\beta|+j-1}(k-j)^{k-j-1}t^{-\mu}.
\end{multline*}
When $M>\max\{M_0,K\}$ and $|\beta|+j\to\infty,$ we have
\[
\frac{M_0^{\frac{|\beta|}{2}+j}K^{\frac{|\beta|+1}{2}+\frac{d}{2}(1-\frac{1}{a})}}{M^{|\beta|+j}}(|\beta|+j)^{2+d(1-\frac{1}{a})}\to 0.
\]
Set 
\begin{align}
h_4(M):=\frac{2CN}{\lambda}\sup_{|\beta|+j\geq1}\left[\frac{M_0^{\frac{|\beta|}{2}+j}K^{\frac{|\beta|+1}{2}+\frac{d}{2}}}{M^{|\beta|+j}}(|\beta|+j)^{2+d}\right],
\end{align}
and then $h_4$ satisfies $\lim_{M\to\infty}h_4(M)=0$. We have
\begin{equation}\label{eq:I2final}
\begin{aligned}
I_2\leq h_4(M)M^{|\beta|+k-\delta}(|\beta|+k)^{|\beta|+k-1}t^{-\mu}. 
\end{aligned}
\end{equation}
Define
\[
h(M):=2[h_1(M)+h_2(M)+h_3(M)+h_4(M)].
\]
Combining \eqref{eq:induction}, \eqref{firstterm}, \eqref{eq:case1}, \eqref{eq:estimateI1}, and \eqref{eq:I2final}, we finally obtain
\begin{align*}
\left\|D_x^\beta\partial_t^k\left(t^k\rho(t)\right)\right\|_{L^q}\leq\frac{1}{2}h(M)M^{|\beta|+k-\delta}(|\beta|+k)^{|\beta|+k-1}t^{-\mu}+\frac{1}{2}\phi(t)t^{-\mu},
\end{align*}
which implies $\phi(t)\leq h(M) M^{|\beta|+k-\delta}(|\beta|+k)^{|\beta|+k-1}$.
Hence, inequality \eqref{eq:Linduction1} holds for $|\beta|+k=L$ when $|\beta|>0$ and $1< q<\frac{d}{d-1}$. We finish the proof of Step 1.

\

\noindent\textbf{Step 2 ($|\beta|>0$, $|\beta|+k=L$, $1< q<\infty$):} 
For $d=1$, Step 1 has already given the results for $1< q<\infty.$ For $d\geq 2$, the proof is similar to the proof of Theorem \ref{thm:regularity1}, and here we will only prove the case for $\frac{d}{d-1}\leq q<\frac{d}{d-2}$. For $\frac{d}{d-2}\leq q< \infty$, repeat the same process similarly.

Notice that in Step 1, the assumption  $1< q<\frac{d}{d-1}$ was only used in the estimate for the term $I_{12}$ in \eqref{eq:estI1}.  
When $\frac{d}{d-1}\leq q<\frac{d}{d-2}$, we use the results in Step 1 to estimate the term $I_{12}$ in \eqref{eq:estI1} as follows. Recall that
\begin{equation}\label{eq:recall}
\begin{aligned}
I_{12}=&\int_{(1-\epsilon)t}^t\left\|\nabla G(\cdot,t-s)\right\|_{L^{a}}\left\|D_x^\beta\partial_s^k\left[s^k\rho(s)\nabla c(s)\right]\right\|_{L^{p}}\di s \\
\leq&A_0\int_{(1-\epsilon)t}^t (t-s)^{-\frac{1}{2}-\frac{d}{2}(1-\frac{1}{a})} \left\|D_x^\beta\partial_s^k\left[s^k\rho(s)\nabla c(s)\right]\right\|_{L^{p}}\di s.
\end{aligned}
\end{equation}
Due to $\frac{d}{d-1}\leq q<\frac{d}{d-2}$, there exist $a$ and $p$ such that
$1\leq a,~p<\frac{d}{d-1},$ $1+\frac{1}{q}=\frac{1}{a}+\frac{1}{p}.$
Due to \eqref{fineestimate2}, we have
\begin{multline*}
\left\|D_x^\beta\partial_s^k\left[s^k\rho(s)\nabla c(s)\right]\right\|_{L^{p}}\leq NM^{|\beta|+k-2\delta}\left(|\beta|+k\right)^{|\beta|+k-1}s^{-\frac{|\beta|+1}{2}-\frac{d}{2}\left(1-\frac{1}{p}\right)}\\
+\left\|D_x^\beta\partial_s^k\left(s^k\rho(s)\right)\nabla c(s)\right\|_{L^{p}}+\left\|\rho(s)D_x^\beta\partial_s^k\left(s^k\nabla c(s)\right)\right\|_{L^{p}}.
\end{multline*}
By H\"older's inequality for $\frac{1}{p_1}+\frac{1}{p_2}=\frac{1}{p}$ and Hardy-Littlewood-Sobolev inequality for $1+\frac{1}{p_2}=\frac{1}{p_3}+\frac{1}{d}$ and $p_2>d$, we have
\begin{multline}\label{eq:I12}
\left\|D_x^\beta\partial_s^k\left[s^k\rho(s)\nabla c(s)\right]\right\|_{L^{p}}\leq NM^{|\beta|+k-2\delta}\left(|\beta|+k\right)^{|\beta|+k-1}s^{-\frac{|\beta|+1}{2}-\frac{d}{2}\left(1-\frac{1}{p}\right)}\\
+\left\|D_x^\beta\partial_s^k\left(t^k\rho(s)\right)\right\|_{L^{p_1}}\left\|\nabla c(s)\right\|_{L^{p_2}}+\left\|\rho(s)\right\|_{L^{p_1}}\left\|D_x^\beta\partial_s^k\left(s^k\nabla c(s)\right)\right\|_{L^{p_2}}\\
\leq NM^{|\beta|+k-2\delta}\left(|\beta|+k\right)^{|\beta|+k-1}s^{-\frac{|\beta|+1}{2}-\frac{d}{2}\left(1-\frac{1}{p}\right)}\\
+A_1\left\|D_x^\beta\partial_s^k\left(s^k\rho(s)\right)\right\|_{L^{p_1}}\left\|\rho(s)\right\|_{L^{p_3}}+A_1\left\|\rho(s)\right\|_{L^{p_1}}\left\|D_x^\beta\partial_s^k\left(s^k\rho(s)\right)\right\|_{L^{p_3}}.
\end{multline}
Due to $1+\frac{1}{p}=\frac{1}{p_1}+\frac{1}{p_3}+\frac{1}{d},$ there exist $p< p_1<\frac{d}{d-1}$ and $1< p_3<\frac{d}{d-1}$ so that we can apply the results obtained in Step 1. 
Apply assumption \eqref{eq:assumptiona} and \eqref{eq:Linduction1} in \eqref{eq:I12} to obtain
\begin{multline*}
\left\|D_x^\beta\partial_s^k\left[s^k\rho(s)\nabla c(s)\right]\right\|_{L^{p}}\leq NM^{|\beta|+k-2\delta}\left(|\beta|+k\right)^{|\beta|+k-1}s^{-\frac{|\beta|+1}{2}-\frac{d}{2}\left(1-\frac{1}{p}\right)}\\
+2AA_1\theta h(M) M^{|\beta|+k-\delta}(|\beta|+k)^{|\beta|+k-1}s^{-\frac{|\beta|+1}{2}-\frac{d}{2}(1-\frac{1}{p})}.
\end{multline*}
Take it into \eqref{eq:recall}, and we obtain the following estimate similar to \eqref{eq:estI12}:
\begin{equation}\label{eq:estm2}
\begin{aligned}
&I_{12}\leq h_3(M)M^{|\beta|+k-\delta}\left(|\beta|+k\right)^{|\beta|+k-1}t^{-\mu}\\
&+2AA_0A_1\theta h(M) M^{|\beta|+k-\delta}(|\beta|+k)^{|\beta|+k-1}\int_{(1-\epsilon)t}^t  s^{-\frac{|\beta|+1}{2}-\frac{d}{2}(1-\frac{1}{p})}
 (t-s)^{-\frac{1}{2}-\frac{d}{2}\left(1-\frac{1}{a}\right)}\di s\\
 \leq &h_3(M)M^{|\beta|+k-\delta}\left(|\beta|+k\right)^{|\beta|+k-1}t^{-\mu}+2AA_0A_1\theta J_1(\e)h(M) M^{|\beta|+k-\delta}(|\beta|+k)^{|\beta|+k-1}t^{-\mu}\\
\leq &\left[h_3(M)+\frac{h(M)}{2}\right]M^{|\beta|+k-\delta}\left(|\beta|+k\right)^{|\beta|+k-1}t^{-\mu},
\end{aligned}
\end{equation}
where we used \eqref{eq:J} in the last step. Combining \eqref{eq:induction}, \eqref{firstterm},  \eqref{eq:case1}, \eqref{eq:estI1}, \eqref{eq:estI11}, \eqref{eq:I2final}, and \eqref{eq:estm2}, we obtain the estimate \eqref{eq:Linduction} for $1< q<\frac{d}{d-2}$. Therefore, we finish the proof of Step 2 for $d=2$. For any $d>2$, we only need repeat the above process again for $d-2$ times (similar to the proof of Theorem \ref{thm:regularity1}).

\

\noindent\textbf{Step 3 ($|\beta|=0$, $k=L$, $1< q<\frac{d}{d-1}$):}
Recall our assumption for induction: assume that there exists a constant $M$ independent of $k$ such that \eqref{eq:Linduction} holds for $|\beta|=0$ and $0< k\leq L-1$ for some positive integer $L$. We are going to prove \eqref{eq:Linduction} for $k=L$. 

As in Step 1, we will estimate the two terms in \eqref{eq:induction} with $|\beta|=0$. Because \eqref{firstterm} holds for any $|\beta|+k>0$, we only need to deal with the second term in \eqref{eq:induction}.
From  \eqref{shijianjisuan}, we have the following estimate for the second term in \eqref{eq:induction}:
\begin{equation}\label{eq:case2}
\begin{aligned}
&\left\|\partial_t^k\left[t^k\int_0^t\nabla G(\cdot,t-s)\ast\left(\rho(s)\nabla c(s)\right)\di s\right]\right\|_{L^q}\\
\leq&  \sum_{j=0}^k\binom{k}{j}\int_0^t\left\|\left[\partial_t^j\left((t-s)^j\nabla G(\cdot,t-s)\right)\right]\ast\left[\partial_s^{k-j}\left(s^{k-j}\rho(s)\nabla c(s)\right)\right]\right\|_{L^q}\di s\\
=&
\int_0^t\left\|\nabla G(\cdot,t-s)\ast\left[\partial_s^{k}\left(s^{k}\rho(s)\nabla c(s)\right)\right]\right\|_{L^q}\di s\\
&+
\sum_{j=1}^{k-1}\binom{k}{j}\int_0^t\left\|\left[\partial_t^j\left((t-s)^j\nabla G(\cdot,t-s)\right)\right]\ast\left[\partial_s^{k-j}\left(s^{k-j}\rho(s)\nabla c(s)\right)\right]\right\|_{L^{q}}\di s\\
&+
\int_0^t\left\|\left[\partial_t^k\left((t-s)^k\nabla G(\cdot,t-s)\right)\right]\ast\left(\rho(s)\nabla c(s)\right)\right\|_{L^q}\di s=:S_1+S_2+S_3.
\end{aligned}
\end{equation}
Using Young's convolution inequality with $1+\frac{1}{q}=\frac{1}{a}+\frac{1}{p}$, $1\leq a,~p\leq q$ and identity \eqref{fineestimate2} to obtain
\begin{equation*}
\begin{aligned}
S_1\leq &\int_0^t\|\nabla G(\cdot,t-s)\|_{L^a}\left\| \partial_s^k(s^k\rho(s)\nabla c(s))\right\|_{L^p}\di s\\
\leq&A_0NM^{k-2\delta}k^{k-1}\int_0^t(t-s)^{-\frac{1}{2}-\frac{d}{2}(1-\frac{1}{a})}s^{-\frac{1}{2}-\frac{d}{2}(1-\frac{1}{p})}\di s\\
&+A_0\int_0^t(t-s)^{-\frac{1}{2}-\frac{d}{2}(1-\frac{1}{a})} \Big[\left\| \partial_s^k\left(s^k\rho(s)\right)\nabla c(s)\right\|_{L^p}+\left\|\rho(s) \partial_s^k\left(s^k\nabla c(s)\right)\right\|_{L^p}\Big]\di s.
\end{aligned}
\end{equation*}
Since $1\leq a,~p\leq q<\frac{d}{d-1}$, we have 
$-\frac{1}{2}-\frac{d}{2}\left(1-\frac{1}{a}\right)>-1,\quad -\frac{1}{2}-\frac{d}{2}\left(1-\frac{1}{p}\right)>-1$ and
\begin{multline}\label{eq:SS}
S_1\leq A_0NBM^{k-2\delta}k^{k-1}t^{-\mu}\\
+2AA_0A_1\theta\int_0^t(t-s)^{-\frac{1}{2}-\frac{d}{2}\left(1-\frac{1}{a}\right)}s^{-\frac{d}{2}\left(-1+\frac{1}{a}+\frac{1}{d}\right)}\left\|\partial_s^k\left(s^k\rho(s)\right)\right\|_{L^q}\di s,
\end{multline}
where $\mu$ is defined by \eqref{eq:mu} with $|\beta|=0$, and $B:=\int_0^1(1-s)^{-\frac{1}{2}-\frac{d}{2}\left(1-\frac{1}{a}\right)}s^{-\frac{1}{2}-\frac{d}{2}\left(1-\frac{1}{p}\right)}\di s$. Define
\begin{align*}
h_5(M):=A_0NBM^{-\delta},
\end{align*}
and then for $\phi$ defined by \eqref{eq:phi}, we have
\begin{equation}\label{eq:S1}
\begin{aligned}
S_1\leq  h_5(M)M^{k-\delta}k^{k-1}t^{-\mu}+2AA_0A_1B\theta t^{-\mu}\cdot \phi(t).
\end{aligned}
\end{equation}

Next, we estimate the term $S_2$ in \eqref{eq:case2}. By Young's convolution inequality with $1+\frac{1}{q}=\frac{1}{a}+\frac{1}{p}$ and $1\leq a,~p\leq q$, we have
\begin{equation*}
\begin{aligned}
S_2\leq \sum_{j=1}^{k-1}\binom{k}{j}\int_0^{t}\left\|\partial_s^j\left((t-s)^j\nabla G(\cdot,t-s)\right)\right\|_{L^{a}}\left\|\partial_s^{k-j}\left(s^{k-j}\rho(s)\nabla c(s)\right)\right\|_{L^{p}}\di s.
\end{aligned}
\end{equation*}
Because $1\leq j\leq k-1$, we  use inequalities \eqref{eq:estimate6} and \eqref{fineestimate1} to obtain
\begin{equation}\label{eq:S21}
\begin{aligned}
S_{2}\leq&\sum_{j=1}^{k-1}\binom{k}{j}M_0^jj^{\frac{1}{2}+j+\frac{d}{2}\left(1-\frac{1}{a}\right)}N\left(1+M^\delta\right)M^{k-j-2\delta}(k-j)^{k-j-1}\\
&\qquad \times\int_0^{t}(t-s)^{-\frac{1}{2}-\frac{d}{2}\left(1-\frac{1}{a}\right)}s^{-\frac{1}{2}-\frac{d}{2}\left(1-\frac{1}{p}\right)}\di s\\
\leq&\left[\sup_{1\leq j\leq k-1}\left(\frac{M_0}{M}\right)^jj^{\frac{3}{2}+\frac{d}{2}\left(1-\frac{1}{a}\right)}\right]\lambda NB\left(M^{-\delta}+1\right)M^{k-\delta}k^{k-1}t^{-\mu}.
\end{aligned}
\end{equation}
Set
\begin{align*}
h_6(M):=2\lambda NB\left[\sup_{j\geq1}\left(\frac{M_0}{M}\right)^jj^{\frac{3}{2}+\frac{d}{2}}\right],
\end{align*}
which implies
\begin{align}\label{eq:S2}
S_2 \leq h_6(M)M^{k-\delta}k^{k-1}t^{-\mu}.
\end{align}

For  the term $S_3$ in \eqref{eq:case2}, we use Hardy-Littlewood-Sobolev inequality again with $1+\frac{1}{q}=\frac{1}{a}+\frac{1}{p}$, $1\leq a,~p\leq q$, and H\"older inequality with $\frac{1}{p}=\frac{1}{p_1}+\frac{1}{p_2}$ to obtain
\begin{align*}
S_3\leq& \int_0^{t}\left\|\left[\partial_t^k\left((t-s)^k\nabla G(\cdot,t-s)\right)\right]\right\|_{L^{a}}\left\|\rho(s)\nabla c(s)\right\|_{L^{p}}\di s\\
\leq&\int_0^{t}M_0^kk^{\frac{1}{2}+k+\frac{d}{2}\left(1-\frac{1}{q}\right)}(t-s)^{-\frac{1}{2}-\frac{d}{2}\left(1-\frac{1}{a}\right)}\|\rho(s)\|_{L^{p_1}}\|\nabla c(s)\|_{L^{p_2}}\di s\\
=& A^2B\theta^2M^\delta\left(\frac{M_0}{M}\right)^k k^{\frac{3}{2}+\frac{d}{2}\left(1-\frac{1}{q}\right)}\cdot M^{k-\delta}k^{k-1}t^{-\frac{d}{2}\left(1-\frac{1}{q}\right)},
\end{align*}
where $A$ comes from \eqref{eq:assumptiona}. Set
\begin{align*}
h_7(M):=A^2B\theta^2\sup_{k\geq1}\left[M^\delta\left(\frac{M_0}{M}\right)^kk^{\frac{3}{2}+\frac{d}{2} }\right],
\end{align*}
and then
\begin{align}\label{eq:S3}
S_3 \leq h_7(M)M^{k-\delta}k^{k-1}t^{-\mu}.
\end{align}
Combining \eqref{eq:induction}, \eqref{firstterm}, \eqref{eq:case2}, \eqref{eq:S1}, \eqref{eq:S2}, and \eqref{eq:S3}, we have the following estimate for the second term in \eqref{eq:induction}:
\begin{equation}\label{eq:case22}
\begin{aligned}
\left\|\partial_t^k(t^k \rho(t))\right\|_{L^q}\leq\tilde{h}(M)M^{k-\delta}k^{k-1}t^{-\mu}+2AA_0A_1B\theta t^{-\mu}\cdot\phi(t),
\end{aligned}
\end{equation}
where $\tilde{h}(M)=h_1(M)+h_5(M)+h_6(M)+h_7(M)$. Moreover, we have $\tilde{h}(M)\to0$ as $M\to\infty$. Therefore,
\begin{align}\label{eq:gronwall}
\phi(t)\leq \tilde{h}(M)M^{k-\delta}k^{k-1}+2AA_0A_1B\theta \cdot\phi(t).
\end{align}
When $\theta$ is small enough, we have
\[
\phi(t)\leq \frac{\tilde{h}(M)}{1-2AA_0A_1B\theta}M^{k-\delta}k^{k-1}.
\]
Hence, inequality \eqref{eq:Linduction2} holds for $\bar{h}(M)=\frac{\tilde{h}(M)}{1-2AA_0A_1B\theta}$.
This finishes the proof of Step 3. 

\

\noindent\textbf{Step 4 ($|\beta|=0$, $k=L$, $1< q< \infty$):}
Since the proof of this part is similar to the proof of Theorem \ref{thm:regularity1} (or the proof of Step 2), we will omit it here.

\

\noindent \textbf{Part II ($q=\infty$ and $q=1$):} We will only sketch the main idea for $q=\infty$ and the proof of $q=1$ is similar. The proof shares some similarity with the proof of the case $q=\infty$ in Theorem \ref{thm:regularity1}.  From Part I, there exists $M>0$ such that
\begin{align}\label{eq:Linduction3}
\left\|D_x^\beta\partial_t^k\left(t^k\rho(t)\right)\right\|_{L^{q}}\leq M^{|\beta|+k}(|\beta|+k)^{|\beta|+k}t^{-\mu},\quad 1< q<\infty.
\end{align}
Since the constant $M$ might depend on $q$, we can not pass to the limit $q\to\infty$ to obtain the estimate for $L^\infty$ norm (also see Remark \ref{rmk:q}).
We are going to use the results for $1<q<\infty$ to prove \eqref{eq:Linduction} for $q=\infty$. We only need to estimate the $L^\infty$ norm for the second term in \eqref{eq:induction}. From \eqref{shijianjisuan}, we have
\begin{align*}
&\left\|D_x^\beta\partial_t^k\left[t^k\int_0^t\nabla G(\cdot,t-s)\ast\left[\rho(s)\nabla c(s)\right]\di s\right]\right\|_{L^\infty}\\
=&\left\|D_x^\beta\sum_{j=0}^k\binom{k}{j}\int_0^t\left[\partial_t^j((t-s)^j\nabla G(\cdot,t-s))\right]\ast\left[\partial_s^{k-j}(s^{k-j}\rho(s)\nabla c(s))\right]\di s\right\|_{L^\infty}\\
\leq&\left\|\sum_{j=0}^k\binom{k}{j}\int_0^{t/2}\left[D_x^\beta\partial_t^j((t-s)^j\nabla G(\cdot,t-s))\right]\ast\left[\partial_s^{k-j}(s^{k-j}\rho(s)\nabla c(s))\right]\di s\right\|_{L^\infty}\\
&+\left\|\sum_{j=0}^k\binom{k}{j}\int_{t/2}^t\left[\partial_t^j((t-s)^j\nabla G(\cdot,t-s))\right]\ast\left[D_x^\beta\partial_s^{k-j}(s^{k-j}\rho(s)\nabla c(s))\right]\di s\right\|_{L^\infty}\\
\leq&\sum_{j=0}^k\binom{k}{j}\int_0^{t/2}\left\|D_x^\beta\partial_t^j((t-s)^j\nabla G(\cdot,t-s))\right\|_{L^\infty}\left\|\partial_s^{k-j}(s^{k-j}\rho(s)\nabla c(s))\right\|_{L^1}\di s \\
&+ \sum_{j=0}^k\binom{k}{j}\int_{t/2}^t\left\|\partial_t^j((t-s)^j\nabla G(\cdot,t-s))\right\|_{L^{\frac{2d}{2d-1}}}\left\|D_x^\beta\partial_s^{k-j}(s^{k-j}\rho(s)\nabla c(s))\right\|_{L^{2d}}\di s.
\end{align*}
We only need to use \eqref{eq:estimate6} and an inequality similar to \eqref{fineestimate1} to estimate the above two terms separately. Then, we could obtain \eqref{eq:Linduction3} for $q=\infty$ with some constant slightly bigger than $M$.

\end{proof}
\begin{remark}[Without assumption \eqref{eq:assumptiona}]\label{rmk:assumptiona}
We can also prove Step 1 and Step 2 in Theorem \ref{thm:analytic} without the assumption \eqref{eq:assumptiona}, which implies the spatial analyticity for any initial data $\rho_0\in L^1(\mathbb{R}^d)$. 

\end{remark}
\begin{remark}[Radius of analyticity]\label{rmk:radius}
For any dimension $d$ and a multi-index $\alpha\in\mathbb{N}^d$ with $|\alpha|=m$, from the polynomial identity $(\sum_{i=1}^dx_i)^m=\sum_{|\gamma|=\alpha}\binom{m}{\gamma}\Pi_{i=1}^dx_i^{\gamma_i}$ for $x_i=1$, we have
\begin{align}\label{eq:factorial}
\alpha !\leq m!\leq d^m\alpha !.
\end{align}
Notice that for initial datum $\rho_0$ with $\|\rho_0\|_{L^1}$ small enough, the constant $M$ is independent of time. Combining \eqref{eq:factorial} and   Stirling's formula, we have the following estimate for the radius $r(t)$ of analyticity  of the mild solution:
\begin{align}\label{eq:radius}
r(t)=\lim_{|\beta|+k\to\infty}\left(\frac{\|D_x^\beta\partial_t^k\rho(t)\|_{L^\infty}}{\beta ! k!}\right)^{-\frac{1}{|\beta|+k}}\geq C\min\{\sqrt{t},t\},\quad t\in(0,\infty).
\end{align}
Similarly, for the space analytic radius $r_1(t)$, we have $r_1(t)\geq C\sqrt{t}$, and for the time analytic radius $r_2(t)$, we have $r_2(t)\geq Ct$ for any $t>0$.
\end{remark}
\begin{remark}\label{rmk:q}
In the proof of \eqref{eq:analytic}, the constant $M$ might depend on $q$. However, we can remove the dependence on $q$ by some more careful estimates. The idea was used to prove $q=\infty$ in Theorem \ref{thm:regularity1}. We can use the results of \eqref{eq:analytic} for $q$ in bounded interval $[1,4d]$ (which yields a uniform constant $M$) to do estimates for all the rest $q>4d$ (as we do for $q=\infty$; see Step 1 and Step 2 in the proof of Theorem \ref{thm:regularity1}).
\end{remark}
\begin{remark}[Global solutions]
In \cite{wei2018global}, Dongyi Wei provided a method to globally extend the mild solutions to the PKS equations for any initial data  $0\leq \rho_0\in L^1(\mathbb{R}^2)$ satisfying $\|\rho_0\|_{L^1}\leq 8\pi$. The method in \cite{wei2018global} can be directly generalized to the mPKS equations  \eqref{eq:KS} for any dimension $d>0$ and global mild solutions will be obtained for any initial data  $0\leq \rho_0\in L^1(\mathbb{R}^d)$ satisfying $\|\rho_0\|_{L^1}\leq 2d^2\pi$. 
\end{remark}

We have the following corollary  (corresponding to Theorem \ref{thm1} (ii)) for bounded mild solutions:
\begin{corollary}\label{cor:allinitial}
Let $0\leq \rho_0\in L^{1}(\mathbb{R}^d)$ and $\rho(t)$ be the mild solution to \eqref{eq:KS} in $[0,T]$.  Let $\beta\in\mathbb{N}^d$ be a multi-index and $k\in\mathbb{N}$ be a nonnegative integer. Assume 
\begin{align}\label{eq:bound}
|\rho(x,t)|\leq C,\quad (x,t)\in\mathbb{R}^d\times[0,T]
\end{align}
for some constant $C$. Then there exists a constant  $M$ depending on $T$ but independent of $\beta$, $k$, such that 
\begin{align}\label{eq:analytic1}
\left\|D_x^\beta\partial_t^k\rho(t)\right\|_{L^q}\leq M^{|\beta|+k}\left(|\beta|+k\right)^{|\beta|+k}t^{-\frac{|\beta|}{2}-k-\frac{d}{2}(1-\frac{1}{q})}
\end{align}
holds for $1\leq q\leq\infty$, $ t\in(0,T]$, and $|\beta|+k>0$.

\end{corollary}

\begin{proof}
The proof shares lots of similarities with the proof of Theorem \ref{thm:analytic}. The main difference comes from  the assumption \eqref{eq:bound}. We will put the proof in Appendix \ref{app:proof}.
 
\end{proof}

\noindent\textbf{Acknowledgements} 
Y. Gao was supported by the Start-Up Fund from the Hong Kong Polytechnic University. X. Xue was supported by the Chinese Natural Science Foundation grants 11731010 and 11671109.

\appendix

\section{Proof of time continuity of $\rho_2(t)$ in \eqref{eq:rho12}}\label{app:continuity}
\begin{proof}
Let $t>\tau>0$ and we have
\begin{multline*}
\|\rho_2(t)-\rho_2(\tau)\|_{L^{1}}\leq \left\|\int_\tau^t\nabla G(\cdot,t-s)\ast(\rho(s)\nabla c(s))\di s\right\|_{L^{1}}\\
+\left\|\int_0^\tau\nabla G(\cdot,t-s)\ast(\rho(s)\nabla c(s))-\nabla G(\cdot,\tau-s)\ast(\rho(s)\nabla c(s))\di s\right\|_{L^1}=:I_1+I_2.
\end{multline*}
For $I_1$, we have
\begin{align*}
I_1=\left\|\int_\tau^t\nabla G(\cdot,t-s)\ast(\rho(s)\nabla c(s))\di s\right\|_{L^1}\leq C\theta^2\int_{\tau/t}^1(1-s)^{-\frac{1}{2}}s^{-\frac{1}{2}}\di s\to 0~\textrm{ as }~t\to\tau.
\end{align*}
For $I_2$, set $g(x,s):=\nabla G(\cdot,\tau-s)\ast(\rho(s)\nabla c(s))$, $0<s<\tau$ and then
\[
\nabla G(\cdot,t-s)\ast(\rho(s)\nabla c(s))=G(\cdot,t-\tau)\ast g(s).
\]
We have
\begin{equation}\label{eq:I2}
\begin{aligned}
I_2=\left\|\int_0^\tau G(\cdot,t-\tau)\ast g(s)-g(s)\di s\right\|_{L^{1}}\leq\int_0^\tau\left\| G(\cdot,t-\tau)\ast g(s)-g(s)\right\|_{L^{1}}\di s.
\end{aligned}
\end{equation}
Next, we estimate the integrand $\left\| G(\cdot,t-\tau)\ast g(s)-g(s)\right\|_{L^{1}}$. For arbitrary $r>0$, we have 
\begin{equation}\label{eq:sindpendent}
\begin{aligned}
&\left\| G(\cdot,t-\tau)\ast g(s)-g(s)\right\|_{L^{1}}\\
=&\int_{\mathbb{R}^d}\left|\int_{\mathbb{R}^d}G(x-y,t-\tau)[g(y,s)-g(x,s)]\di y\right|\di x\\
\leq&\int_{\mathbb{R}^d}\int_{\mathbb{R}^d}G(x-y,t-\tau)|g(y,s)-g(x,s)|\di y\di x\\
\leq &\int_{\mathbb{R}^d}\int_{B(x,r)}G(x-y,t-\tau)|g(y,s)-g(x,s)|\di y\di x\\
&+\int_{\mathbb{R}^d}\int_{\mathbb{R}^d\setminus B(x,r)}G(x-y,t-\tau)|g(y,s)-g(x,s)|\di y\di x=: I_{21}+I_{22}.
\end{aligned}
\end{equation}
For the first integral in \eqref{eq:sindpendent}, we have
\begin{equation}\label{eq:firstterm}
\begin{aligned}
I_{21}=&\int_{\mathbb{R}^d}\int_{B(x,r)}G(x-y,t-\tau)|g(y,s)-g(x,s)|\di y\di x\\
=&\int_{\mathbb{R}^d}\int_{B(0,r)}G(z,t-\tau)|g(x+z,s)-g(x,s)|\di z\di x\\
= &\int_{B(0,r)}G(z,t-\tau)\int_{\mathbb{R}^d}|g(x+z,s)-g(x,s)|\di x\di z\\
\leq& \int_{B(0,r)}G(z,t-\tau)\sup_{h\in\mathbb{R}^d,|h|<r}\int_{\mathbb{R}^d}|g(x+h,s)-g(x,s)|\di x\di z\\
\leq&\sup_{h\in\mathbb{R}^d,|h|<r}\|g(\cdot+h,s)-g(\cdot,s)\|_{L^{1}}.
\end{aligned}
\end{equation}
Notice that $\nabla G(x,t)=t^{-\frac{d+1}{2}}\nabla G\left(t^{-\frac{1}{2}}x,1\right)$. Denote
\[
f_h(x)=(\tau-s)^{-\frac{d+1}{2}}\left|\nabla G\big((\tau-s)^{-\frac{1}{2}}(x+h),1\big)-\nabla G\big((\tau-s)^{-\frac{1}{2}}x,1\big)\right|.
\]
By the definition of $g$, we have
\begin{equation}\label{eq:sindependet1}
\begin{aligned}
I_{21}\leq& \sup_{|h|\leq r}\|f_h\ast (\rho(s)\nabla c(s))\|_{L^{1}}\leq \sup_{|h|\leq r}\|f_h\|_{L^1}\|\rho(s)\nabla c(s)\|_{L^{1}}\\	\leq&C\theta^2(\tau-s)^{-\frac{1}{2}}s^{-\frac{1}{2}}\sup_{|h|\leq r}\left\| \nabla G(\cdot+h,1)-\nabla G(\cdot,1)\right\|_{L^1}.
\end{aligned}
\end{equation}
Due to $G(x,t)=t^{-\frac{d}{2}}G(t^{-\frac{1}{2}}x,1)$, we obtain
\begin{equation}\label{eq:secondterm}
\begin{aligned}
I_{22}=&\int_{\mathbb{R}^d}\int_{\mathbb{R}^d\setminus B(x,r)}G(x-y,t-\tau)|g(y,s)-g(x,s)|\di y\di x\\
\leq& \int_{\mathbb{R}^d}\int_{\mathbb{R}^d\setminus B(0,r/(t-\tau)^{\frac{1}{2}})}G(z,1)|g(x-(t-\tau)^{\frac{1}{2}}z,s)-g(x,s)|\di z\di x\\
\leq &2\|g(s)\|_{L^{1}}\int_{\mathbb{R}^d\setminus B(0,r/(t-\tau)^{\frac{1}{2}})}G(z,1)\di z.
\end{aligned}
\end{equation}
By $\|\nabla G(\cdot,t)\|_{L^1}=t^{-\frac{1}{2}}\|\nabla G(\cdot,1)\|_{L^1}$
and Young's convolution inequality, we obtain
\begin{equation}\label{eq:sindependet0}
\begin{aligned}
\|g(s)\|_{L^{1}}\leq \|\nabla G(\cdot,\tau-s)\|_{L^1}\|\rho(s)\nabla c(s)\|_{L^{1}}\leq C\theta^2(\tau-s)^{-\frac{1}{2}}s^{-\frac{1}{2}}.
\end{aligned}
\end{equation}
Combining \eqref{eq:I2}-\eqref{eq:sindependet0}, we obtain
\begin{equation}
\begin{aligned}
I_2
\leq &C\theta^2\int_0^\tau (\tau-s)^{-\frac{1}{2}}s^{-\frac{1}{2}}\di s\sup_{|h|\leq r}\|\nabla G(\cdot+h,1)-\nabla G(\cdot,1)\|_{L^1}\\
&\qquad\qquad+C\theta^2\int_0^\tau (\tau-s)^{-\frac{1}{2}}s^{-\frac{1}{2}}  \di s  \int_{\mathbb{R}^d\setminus B(0,r/(t-\tau)^{\frac{1}{2}})}G(z,1)\di z.
\end{aligned}
\end{equation}
By \cite[Lemma 4.3]{brezis}, send $t\to\tau$ first and then $r\to0$, we have $I_1+I_2\to0$.
\end{proof}

\section{Proof of Corollary \ref{cor:allinitial}}\label{app:proof}

\begin{proof}[Proof of Corollary \ref{cor:allinitial}] 
Although we do not have the assumption \eqref{eq:assumptiona}, the same calculations show that Step 1 and Step 2 in the proof of Theorem \ref{thm:analytic} still hold.
Therefore, we only need to revise Step 3, i.e., time analyticity ($\beta=0$ in \eqref{eq:analytic1}) in the proof of Theorem \ref{thm:analytic}. 
Similarly, we will use induction to prove that there exists a constant $M$ depending on $T$ but independent of $k$  such that 
\begin{align}\label{eq:assume}
\left\|\partial_t^k\left(t^k\rho(t)\right)\right\|_{L^q}\leq M^{k-\delta}k^{k-1},\quad k\geq 1,\quad 1<q<\infty,\quad t\in[0,T].
\end{align}
For the cases $q=\infty$ and $q=1$, we can use the same proof as Part II in Theorem \ref{thm:analytic}. Then, combining Step 1, Step 2 in Theorem \ref{thm:analytic} and the above estimates \eqref{eq:assume}, we can obtain \eqref{eq:analytic1} immediately in a bounded time interval $(0,T]$. 

Due to the assumption \eqref{eq:bound} and the conservation of mass $\|\rho(t)\|_{L^1}$, we have
\begin{align}\label{eq:bounds}
\sup_{1\leq q\leq\infty}\left\|\rho(t)\right\|_{L^q}\leq C_1:=\max\{\|\rho(t)\|_{L^\infty},~\|\rho(t)\|_{L^1}\}.
\end{align}
Next, we separate the proof into two parts. In Part I, we show the first step of induction for \eqref{eq:assume}, i.e., the case $k=1$. In Part II, we prove \eqref{eq:assume} for any $k\geq1.$

\

\noindent \textbf{Part I ($k=1$ and $1<q<\infty$)} In this part, we are going to prove \eqref{eq:assume} for $k=1$ and $1<q<\infty$. We separate this part into two steps. In the first step, we prove that there exist a constant $C_T$ depending on $T$ such that
\begin{align}\label{eq:result1}
\left\|\partial_t\left(t\rho(t)\right)\right\|_{L^q}\leq C_T,\quad 1< q<\frac{d}{d-1},\quad t\in[0,T].
\end{align}
In the second step, we show that \eqref{eq:result1} holds for $1<q<\infty$.

\textbf{Step 1.} In this step, we prove \eqref{eq:result1} for $1<q<\frac{d}{d-1}$. Direct calculation shows
\begin{equation}\label{eq:c1}
\begin{aligned}
\left\|\partial_t\left(t\rho(t)\right)\right\|_{L^q}\leq&\left\|\partial_t\left(tG(\cdot,t)\ast\rho_0\right)\right\|_{L^q}+\left\|\partial_t\left[t\int_0^t\nabla G(\cdot,t-s)\ast[\rho(s)\nabla c(s)]\right]\di s\right\|_{L^q}\\
=&:I_1+I_2.
\end{aligned}
\end{equation}
For $I_1$, by Young's inequality and \eqref{eq:estimate5}, we have 
\begin{align}\label{eq:c2}
I_1\leq\left\|\partial_t(tG(\cdot,t))\right\|_{L^1}\|\rho_0\|_{L^q}\leq M_0C_1.
\end{align}
For $I_2$, we use \eqref{shijianjisuan} to obtain
\begin{multline}\label{eq:c3}
I_2\leq\int_0^t\left\|\nabla G(\cdot,t-s)\ast\left[\partial_s(s\rho(s)\nabla c(s))\right]\right\|_{L^q}\di s\\
+\int_0^t\left\|\left[\partial_t\left((t-s)\nabla G(\cdot,t-s)\right)\right]\ast[\rho(s)\nabla c(s)]\right\|_{L^q}\di s=:I_{21}+I_{22}.
\end{multline}
Using Young's inequality with $1+\frac{1}{q}=\frac{1}{a}+\frac{1}{p}$ and $1\leq a,~p\leq q<\frac{d}{d-1}$, we can obtain
\begin{equation}\label{eq:c4}
\begin{aligned}
I_{21}=&\int_0^t\left\|\nabla G(\cdot,t-s)\ast\left[\partial_s(s\rho(s)\nabla(s))\right]\right\|_{L^q}\di s\\
\leq&\int_0^t\left\|\nabla G(\cdot,t-s)\right\|_{L^a}\left\|\partial_s(s\rho(s)\nabla(s))\right\|_{L^p}\di s\\
\leq&\int_0^t(t-s)^{-\mu_a}\left(\left\|\rho(s)\partial_s(s\nabla c(s))\right\|_{L^p}\left\|\partial_s(s\rho(s))\nabla c(s)\right\|_{L^p}+\left\|\rho(s)\nabla c(s)\right\|_{L^p}\right)\di s\\
\leq&\frac{C_1^2}{1-\mu_a}T^{1-\mu_a}+\int_0^t(t-s)^{-\mu_a}\|\rho(s)\|_{L^{\frac{1}{2-\frac{1}{a}-\frac{1}{d}}}}\|\partial_s(s\nabla c(s))\|_{L^{\frac{1}{\frac{1}{q}+\frac{1}{d}-1}}}\di s\\
&\qquad\qquad\qquad\qquad\qquad\qquad+\int_0^t(t-s)^{-\mu_a}\|\partial_s(s\rho(s))\|_{L^q}\|\nabla c(s)\|_{L^{\frac{1}{\frac{1}{p}-\frac{1}{q}}}}\di s\\
\leq&\frac{C_1^2}{1-\mu_a}T^{1-\mu_a}+2A_1C_1\int_0^t(t-s)^{-\mu_a}\|\partial_s(s\rho(s))\|_{L^q}\di s,
\end{aligned}
\end{equation}
where $A_1$ is a constant comes from Hardy-Littlewood-Sobolev inequality and $\mu_a:=-\frac{1}{2}-\frac{d}{2}\left(1-\frac{1}{a}\right)$. Use the similar estimate, and we have
\begin{align}\label{eq:c5}
I_{22}\leq\frac{A_1C_1^2M_0}{1-\mu_a}T^{1-\mu_a}.
\end{align}
Denote
\begin{align*} 
\varphi(t):=\|\partial_t(t\rho(t))\|_{L^q}.
\end{align*}
Combining inequalities \eqref{eq:c1}, \eqref{eq:c2}, \eqref{eq:c3}, \eqref{eq:c4} and \eqref{eq:c5}, we have the following estimate
\begin{align*}
\varphi(t)\leq H+2A_1C_1\int_0^t(t-s)^{-\mu_a}\varphi(s)\di s,
\end{align*}
where $H:=M_0C_1+\frac{C_1^2}{1-\mu_a}T^{1-\mu_a}+\frac{A_1C_1^2M_0}{1-\mu_a}T^{1-\mu_a}$ is a constant. We get by a second iteration and using Fubini's theorem that
\begin{equation}\label{eq:c6}
\begin{aligned}
\varphi(t)\leq&H+2A_1C_1\int_0^t(t-s)^{-\mu_a}\left[H+2A_1C_1\int_0^s(s-\tau)^{-\mu_a}\varphi(\tau)\di \tau\right]\di s\\
\leq&\left(H+\frac{2A_1C_1H}{1-\mu_a}T^{1-\mu_a}\right)+4A_1^2C_1^2\int_0^t(t-s)^{-\mu_a}\int_0^s(s-\tau)^{-\mu_a}\varphi(\tau)\di \tau\di s\\
=&\left(H+\frac{2A_1C_1H}{1-\mu_a}T^{1-\mu_a}\right)+4A_1^2C_1^2\int_0^t\varphi(\tau)\int_{\tau}^t(t-s)^{-\mu_a}(s-\tau)^{-\mu_a}\di s\di \tau\\
=&\left(H+\frac{2A_1C_1H}{1-\mu_a}T^{1-\mu_a}\right)+4A_1^2C_1^2\mathcal{B}(1-\mu_a,1-\mu_a)\int_0^t(t-\tau)^{1-2\mu_a}\varphi(\tau)\di \tau\\
=&:H_1^{(1)}+H_2^{(1)}\int_0^t(t-\tau)^{1-2\mu_a}\varphi(\tau)\di \tau.
\end{aligned}
\end{equation}
Repeat the process of inequality \eqref{eq:c6} and at the $i$th time, we get
\begin{align*}
\varphi(t)\leq H_1^{(i)}+H_2^{(i)}\int_0^t(t-s)^{2^i(1-\mu_a)-1}\varphi(s)\di s.
\end{align*}
Recall that $\frac{1}{2}<\mu_a<1$, then there exists $i_0$ big enough such that $2^{i_0}(1-\mu_a)-1\geq 0$. Because $T\geq t-s$, we have
\begin{align*}
\varphi(t)\leq H_1^{(i_0)}+H_2^{(i_0)}T^{2^{i_0}(1-\mu_a)-1}\int_0^t\varphi(s)\di s.
\end{align*}
Using Gr\"onwall inequality, we obtain
\begin{align*}
\varphi(t)\leq C_T:=H_1^{(i_0)}\exp\{H_2^{(i_0)}T^{2^{i_0}(1-\mu_a)}\}.
\end{align*}

\textbf{Step 2.}  The proof of inequality \eqref{eq:result1} for $\frac{d}{d-1}\leq q<\infty$ is similar to the proof of Theorem \ref{thm:regularity1} (or Step 2 of the proof of Theorem \ref{thm:analytic}), and we omit it. 

\

\noindent \textbf{Part II ($k=L\geq 2$ and $1<q<\infty$):} In this part, we assume that there exists a constant $M$ (to be fixed) independent of $k$ such that \eqref{eq:assume} holds for $0 <k\leq L-1$ for some $L\geq2$. We are going to prove that \eqref{eq:assume} also holds for $k = L$ with the same $M$. 

\textbf{Step 1.}
In this step, we assume $k=L$, and we are going to prove 
\begin{align}\label{eq:hm}
\left\|\partial_t^k\left(t^k\rho(t)\right)\right\|_{L^q}\leq \bar{h}_0(M)M^{k-\delta}k^{k-1}~\textrm{ for }~1<q<\frac{d}{d-1},
\end{align}
where $\bar{h}_0$ is independent of $q$ and $k$, and satisfies $\bar{h}_0(M)\to 0$ as $M\to \infty$.
Inequality \eqref{eq:case2} shows that
\begin{equation}\label{eq:direct}
\begin{aligned}
\left\|\partial_t^k\left(t^k\rho(t)\right)\right\|_{L^q}\leq&\left\|\partial_t^k\left[t^kG(\cdot,t)\ast\rho_0\right]\right\|_{L^q}+\left\|\partial_t^k\left[t^k\int_0^t\nabla G(\cdot,t-s)\ast\left(\rho(s)\nabla c(s)\right)\di s\right]\right\|_{L^q}\\
\leq&\left\|\partial_t^k\left[t^kG(\cdot,t)\ast\rho_0\right]\right\|_{L^q}+S_1+S_1+S_3.
\end{aligned}
\end{equation}
Using inequality \eqref{eq:estimate5} and Young's inequality, we obtain
\begin{align*}
\left\|\partial_t^k\left[t^kG(\cdot,t)\ast\rho_0\right]\right\|_{L^q}\leq\left\|\partial_t^k\left(t^kG(\cdot,t)\right)\right\|_{L^1}\left\|\rho_0\right\|_{L^q}\leq C_1M^{\delta}\left(\frac{M_0}{M}\right)^kkM^{k-\delta}k^{k-1}.
\end{align*}
For any $M>M_0$, notice that $C_1M^{\delta}\left(\frac{M_0}{M}\right)^kk\to 0$ as $k\to\infty$. We define 
\begin{align*}
h_8(M):=\sup_{k\geq1}{C_1M^{\delta}\left(\frac{M_0}{M}\right)^kk}.
\end{align*}
Then, we have
\begin{align}\label{eq:s0term}
\left\|\partial_t^k\left[t^kG(\cdot,t)\ast\rho_0\right]\right\|_{L^q}\leq h_8(M)M^{k-\delta}k^{k-1}.
\end{align}
For $n\leq L-1$ and $k=L$, we have the following  inequalities similar to \eqref{fineestimate1} and \eqref{fineestimate2}:
\begin{align}\label{eq:ee1}
\left\|\partial_t^n\left[t^n\rho(t)\nabla c(t)\right]\right\|_{L^p}\leq C_2(1+M^{\delta})M^{n-2\delta}n^{n-1}
\end{align}
and
\begin{align}\label{eq:ee2}
\left\|\partial_t^k\left[t^k\rho(t)\nabla c(t)\right]\right\|_{L^p}\leq\left\|\rho(t)\partial_t^k\left(t^k\nabla c(t)\right)\right\|_{L^p}+\left\|\partial_t^k\left(t^k\rho(t)\right)\nabla c(t)\right\|_{L^p}+C_2M^{k-2\delta}k^{k-1},
\end{align}
where $C_2$ is a constant independent of $M$, $q$ and $k$.
Next, we obtain estimates for $S_1$, $S_2$ and $S_3$ in \eqref{eq:direct}. Using similar method as Step 3 of the proof of Theorem \ref{thm:analytic} with $1+\frac{1}{q}=\frac{1}{a}+\frac{1}{p}$, $1<a,p<q<\frac{d}{d-1}$ and $\|\rho(t)\|_{L^q}\leq C_1$ for $1\leq q\leq\infty$, we have
\begin{equation}\label{eq:s1term}
\begin{aligned}
S_1\leq& h_9(M)M^{k-\delta}k^{k-1}+2A_1C_1\int_0^t(t-s)^{-\mu_a}\left\|\partial_s^k\left(s^k\rho(s)\right)\right\|_{L^q}\di s,
\end{aligned}
\end{equation}
\begin{align}\label{eq:s2term}
S_2\leq h_{10}(M)M^{k-\delta}k^{k-1},\qquad S_3\leq h_{11}(M)M^{k-\delta}k^{k-1}
\end{align}
with the three coefficients $h_9(M)=\frac{C_2}{1-\mu_a}T^{1-\mu_a}M^{-\delta}$, $h_{10}=\frac{2C_2}{1-\mu_a}T^{1-\mu_a}[\sup_{j\geq1}\left(\frac{M_0}{M}\right)^jj^{\frac{3}{2}+\frac{d}{2}}]$ and $h_{11}=\frac{A_1C_1^2}{1-\mu_a}T^{1-\mu_a}[\sup_{k\geq1}\left(\frac{M_0}{M}\right)^kk^{\frac{3}{2}+\frac{d}{2}}]$.
Here $\frac{1}{1-\mu_a}T^{1-\mu_a}$ comes from the integral $\int_0^t(t-s)^{-\mu_a}\di s.$
Notice that $h_i(M)\to 0 ~~(i=8,9,10,11)$ as $M\to\infty$.
Define
\begin{align*}
h_0(M):=h_8(M)+h_9(M)+h_{10}(M)+h_{11}(M)
\end{align*}
and 
\begin{align*}
\psi(t):=\left\|\partial_t^k\left(t^k\rho(t)\right)\right\|_{L^q}.
\end{align*}
Combining inequalities \eqref{eq:direct} \eqref{eq:s0term}, \eqref{eq:s1term}, and \eqref{eq:s2term}, we finally obtain
\begin{align*}
\psi(t)\leq h_0(M)M^{k-\delta}k^{k-1}+2A_1C_1\int_0^t(t-s)^{-\mu_a}\psi(s)\di s.
\end{align*}
Using the similar iteration \eqref{eq:c6} of Part I for $i_0$ times, we obtain
\begin{align}
\psi(t)\leq h_0(M)H_3^{(i_0)}M^{k-\delta}k^{k-1}+H_4^{(i_0)}T^{2^{i_0}(1-\mu_a)-1}\int_0^t\psi(s)\di s.
\end{align}
where $H_3^{(i_0)}$ and $H_4^{(i_0)}$ are constants independent of $M$, $q$ and $k$. Then, by Gr\"onwall inequality, we have \eqref{eq:hm} for $\bar{h}_0(M)=H_3^{(i_0)}\exp\{H_4^{(i_0)}T^{2^{i_0}(1-\mu_a)}\}h_0(M)$.

\textbf{Step 2.} The proof of inequality \eqref{eq:assume} for $1<q<\infty$ is similar to the proof of Theorem \ref{thm:regularity1} (or Step 2 of the proof of Theorem \ref{thm:analytic}), and we omit it.

\end{proof}

\bibliographystyle{plain}
\bibliography{bibofKS}

\end{document}